\documentclass[twocolumn,abstract,10pt]{scrartcl}

\usepackage[a4paper, left=1.6cm, right=1.6cm, top=2.5cm, bottom=3.7cm]{geometry}
\usepackage[english]{babel}
\usepackage{amsmath}
\usepackage{amssymb}
\usepackage{amsthm}
\usepackage{enumitem}
\usepackage[colorlinks = true, linkcolor = blue, citecolor = blue]{hyperref}
\usepackage[nameinlink]{cleveref}
\usepackage{graphicx}

\setcapindent{0pt}
\setkomafont{caption}{\itshape}
\setkomafont{captionlabel}{\bfseries}

\theoremstyle{plain}
\newtheorem{thm}{Theorem}

\newtheorem{lem}{Lemma}

\newtheorem{assum}{Assumption}
\newtheorem{prop}{Proposition}

\theoremstyle{definition}
\newtheorem{defn}{Definition}
\newtheorem{exmp}{Example}
\newtheorem{rem}{Remark}

\crefname{equation}{}{}
\crefname{assum}{Assumption}{Assumptions}
\crefname{defn}{Definition}{Definitions}
\crefname{exmp}{Example}{Examples}
\crefname{lem}{Lemma}{Lemmas}
\crefname{notation}{Notation}{Notation}
\crefname{prop}{Proposition}{Propositions}
\crefname{rem}{Remark}{Remarks}
\crefname{thm}{Theorem}{Theorems}

\addto\captionsenglish{}

\title{Robustness properties of a large-amplitude, high-frequency extremum seeking control scheme\footnote{This research was supported by the German Research Foundation DFG, project number DA~767/13-1. Corresponding author: Raik~Suttner.}}
\author{Raik Suttner\footnote{University of Wuerzburg, Wuerzburg, Germany; \texttt{raik.suttner @mathematik.uni- wuerzburg.de}}, $\quad$ Sergey Dashkovskiy\footnote{University of Wuerzburg, Wuerzburg, Germany; \texttt{sergey.dash- kovskiy@mathematik.uni-wuerzburg.de}}}
\date{}

\begin{document}
\maketitle
\begin{abstract}
We analyze stability and robustness properties of an extremum seeking scheme that employs oscillatory dither signals with sufficiently large amplitudes and frequencies. Our study takes both input and output disturbances into account. We consider general~$L_\infty$-disturbances, which may resonate with the oscillatory dither signals. A suitable change of coordinates followed by an averaging procedure reveals that the closed-loop system approximates the behavior of an averaged system. This leads to the effect that stability and robustness properties carry over from the averaged system to the closed-loop system. In particular, we show that, if the averaged system is input-to-state stable (ISS), then the closed-loop system has ISS-like properties.
\end{abstract}


\section{Introduction}\label{sec:01}
Research on extremum seeking control has lead to a variety of new methods and techniques for optimization problems that only allow real-time measurements of an objective (or cost) function~\cite{AriyurBook,LiuBook,ScheinkerBook,ZhangBook}. For example, there are methods based on sliding mode control~\cite{Pan2003}, parameter estimation techniques~\cite{Guay2003}, or numerical optimization~\cite{Khong20132}. Many studies are motivated by practical applications, such as optimization of bio-processes~\cite{Wang1999}, maximum power point tracking of photovoltaic systems~\cite{Leyva2006}, ABS control~\cite{Drakunov1995}, or optimal cam timing~\cite{Popovic2006}. This focus on applications naturally leads to the question of robustness against disturbances. Most of the theoretical results assume ideal and undisturbed implementations of the proposed extremum seeking methods. In practice, however, such perfect conditions are difficult to realize. For instance, disturbances may occur in terms of external forces, measurement errors, or quantized inputs. It is therefore desirable to derive theoretical results that take disturbances into account. The present paper contains such an analysis.

For many extremum seeking control laws, it is difficult to provide a mathematically rigorous proof of robustness. Some studies discuss the influence of disturbances on a rather qualitative level and by numerical simulations; e.g., in~\cite{Zhang2009,Fu2011}. Quantitative statements about robustness usually require suitable assumptions on the disturbances. For example, in~\cite{Ye2016}, robustness of a numerical optimization-based extremum seeking scheme is investigated under the assumption that the disturbances are twice continuously differentiable with uniformly bounded first and second derivative. The problem of robustness becomes especially difficult in perturbation-based schemes, where oscillatory dither signals are fed in to probe the response of the objective function value. In this cases, resonances between the disturbances and the dither signals can lead to a complete loss of stability. The existing studies avoid the difficulty of undesired resonances by imposing suitable assumptions. For example, the main result in~\cite{Tan2010} ensures robustness of the proposed perturbation-based scheme under the assumption that the disturbances and the dither signals are uncorrelated. For the discrete-time perturbation-based scheme in~\cite{Stankovic2010}, resonances are prevented by the assumption that the disturbances take the form of a martingale difference sequence.

In the present paper, we focus on a perturbation-based extremum seeking scheme that employs oscillatory dither signals with sufficiently large amplitudes and frequencies. This scheme can be seen as a generalization of the control laws in~\cite{Zhang20071,Zhang20072}. Note that the analysis in~\cite{Zhang20071,Zhang20072} does not take external disturbances into account. In contrast to the closely-related schemes in~\cite{Krstic2000,Tan2006}, the proposed approach does not rely on internal stability properties of the control system but can be applied to potentially unstable systems. The strong dither signals have the purpose to overpower unstable dynamics and to force the system towards an extremum of the objective function. On the other hand, as indicated in the previous paragraph, undesired resonances between disturbances and the dither signals may have a harmful effect on the performance of the closed-loop system. So, at first glance, it seems to be unlikely that such a large-amplitude, high-frequency scheme should have strong robustness properties. However, after a suitable change of coordinates, we are able to prove that the closed-loop can even tolerate oscillatory disturbances with the same frequencies as the dither signals.

There is also a close connection between the extremum seeking scheme in the present paper and the class of large-amplitude, high-frequency control laws in~\cite{Duerr2013,Scheinker2013}. After a suitable change of coordinates, we can show that the closed-loop system approximates the behavior of an averaged system. It turns out that the averaged system is the same Lie bracket system as in~\cite{Duerr2013}; i.e., the averaged system is determined by Lie brackets of vector fields from the closed-loop system. In particular, this establishes a direct link between the methods in~\cite{Zhang20071,Zhang20072} and the methods in~\cite{Duerr2013,Scheinker2013}. In~\cite{Scheinker2016} robustness is shown for a similar scheme as in~\cite{Duerr2013} under the assumption that the disturbances are differentiable with uniformly bounded derivatives. In reality, however, disturbances occur as discontinuous functions and they may vary arbitrary fast. The key difference between the method in this paper and the method in~\cite{Duerr2013} is that the disturbances are not amplified by the dither signals. This beneficial feature allows us to treat general~$L_{\infty}$-disturbances in the analysis for the first time.

The paper is organized as follows. We start in \Cref{sec:02} by recalling some basic definitions and notation from differential geometry. In \Cref{sec:03}, we introduce suitable notions of stability and robustness for the closed-loop system. \Cref{sec:04} provides the tools that are used in \Cref{sec:05} to analyze the extremum seeking system. Our results are illustrated by examples in \Cref{sec:06}.


\section{Notation and Definitions}\label{sec:02}
\hypertarget{sumConvention}{\textbf{Summation convention.}} We use the convention that components of vectors are indexed with \emph{super}scripts, while lists of vectors are indexed with \emph{sub}scripts. Whenever an expression contains a repeated index, one as a subscript and the other as a superscript, summation is implied over this index. Points of a submanifold of a Euclidean space are written as row vectors and tangent vectors are written as column vectors.

Let~$\mathbb{N}$, $\mathbb{R}$, $\mathbb{R}_+$, and~$\bar{\mathbb{R}}_+$ denote the sets of positive integers, real numbers, positive real numbers, and nonnegative real numbers, respectively. For every~$n\in\mathbb{N}$, let~$\langle\cdot,\cdot\rangle$ denote the Euclidean inner product on~$\mathbb{R}^n$ and let~$|\cdot|$ denote the Euclidean norm on~$\mathbb{R}^n$. For every~$n\in\mathbb{N}$, let~$L_{\infty}^n$ denote the set of measurable and locally essentially bounded maps from~$\mathbb{R}$ to~$\mathbb{R}^n$. For every~$w\in{L_{\infty}^n}$, let~$\|w\|$ denote the essential supremum norm of~$w$.

Let~$\mathcal{K}$ denote the set of continuous, strictly increasing functions~$\gamma\colon\bar{\mathbb{R}}_+\to\bar{\mathbb{R}}_+$ with~$\gamma(0)=0$. Let~$\mathcal{K}_{\infty}$ denote the set of~$\gamma\in\mathcal{K}$ with~$\gamma(s)\to\infty$ as~$s\to\infty$. Let~$\mathcal{K}\mathcal{L}$ denote the set of~$\beta\colon\bar{\mathbb{R}}_+\times\bar{\mathbb{R}}_+\to\bar{\mathbb{R}}_+$ such that~$\beta(\cdot,t)\in\mathcal{K}$ for every~$t\in\bar{\mathbb{R}}_+$ and such that~$\beta(s,\cdot)\colon\bar{\mathbb{R}}_+\to\bar{\mathbb{R}}_+$ is decreasing with~$\beta(s,t)\to0$ as~$t\to\infty$ for every~$s\in\bar{\mathbb{R}}_+$.

Let~$\text{\sffamily{\upshape{M}}}$ be a closed (smooth) submanifold of a Euclidean space. Note that the Euclidean norm on the endowing Euclidean space turns~$\text{\sffamily{\upshape{M}}}$ into a metric space, and that every closed and bounded subset of~$\text{\sffamily{\upshape{M}}}$ is compact. For every~$\xi\in\text{\sffamily{\upshape{M}}}$ and every nonempty and compact~$K\subset\text{\sffamily{\upshape{M}}}$, let~$|\xi|_K$ denote the distance of~$\xi$ to~$K$ with respect to the Euclidean norm on the endowing Euclidean space. 

The reader is referred to~\cite{BulloBook} for basic definitions and properties of smooth vector fields and their flows. The word ``smooth'' always means of class~$C^{\infty}$. The assumption of smoothness is just a matter of convenience. The reader is invited to check that most of the constructions in this paper go through for twice continuously differentiable state dependent maps. Let~$\mathfrak{F}(\text{\sffamily{\upshape{M}}})$ denote the set of \emph{smooth real-valued functions} on~$\text{\sffamily{\upshape{M}}}$. Let~$\mathfrak{X}(\text{\sffamily{\upshape{M}}})$ denote the set of \emph{smooth vector fields} on~$\text{\sffamily{\upshape{M}}}$. Every~$X\in\mathfrak{X}(\text{\sffamily{\upshape{M}}})$ can be considered as a linear map from~$\mathfrak{F}(\text{\sffamily{\upshape{M}}})$ to~$\mathfrak{F}(\text{\sffamily{\upshape{M}}})$ that assigns to each~$f\in\mathfrak{F}(\text{\sffamily{\upshape{M}}})$ the \emph{Lie derivative}~$Xf\in\mathfrak{F}(\text{\sffamily{\upshape{M}}})$ of~$f$ along~$X$. For all~$X,Y\in\mathfrak{X}(\text{\sffamily{\upshape{M}}})$, the \emph{Lie bracket} of~$X,Y$ is the unique element~$[X,Y]\in\mathfrak{X}(\text{\sffamily{\upshape{M}}})$ such that
\begin{equation*}
[X,Y]f \ = \ X(Yf) - Y(Xf)
\end{equation*}
for every~$f\in\mathfrak{F}(\text{\sffamily{\upshape{M}}})$. For every smooth diffeomorphism~$\Phi\colon\text{\sffamily{\upshape{M}}}\to\text{\sffamily{\upshape{M}}}$ and every~$X\in\mathfrak{X}(\text{\sffamily{\upshape{M}}})$, the \emph{pull-back} of~$X$ by~$\Phi$ is the unique element~$\Phi^{\ast}X\in\mathfrak{X}(\text{\sffamily{\upshape{M}}})$ such that
\begin{equation*}
(\Phi^{\ast}X)f \ = \ X(f\circ\Phi^{-1})\circ\Phi
\end{equation*}
for every~$f\in\mathfrak{F}(\text{\sffamily{\upshape{M}}})$. For every~$X\in\mathfrak{X}(\text{\sffamily{\upshape{M}}})$, let~$(t,\xi)\mapsto\Phi^{X}_{t}(\xi)$ denote the \emph{flow} of~$X$. A \emph{time-dependent vector field on~$\text{\sffamily{\upshape{M}}}$} is a map that assigns to each pair~$(t,\xi)$ of~$\mathbb{R}\times\text{\sffamily{\upshape{M}}}$ a tangent vector to~$\text{\sffamily{\upshape{M}}}$ at~$\xi$. For every time-dependent vector field~$X$ on~$\text{\sffamily{\upshape{M}}}$ and every~$t\in\mathbb{R}$, let~$X_t$ denote the vector field on~$\text{\sffamily{\upshape{M}}}$ that is defined by~$X_t(\xi):=X(t,\xi)$.


\section{Stability notions}\label{sec:03}
In this section, we introduce a suitable terminology to describe robustness and stability properties of a parameter- and time-dependent system (denoted by~$\tilde{\Sigma}^{\omega}$) that approximates the trajectories of an averaged system (denoted by~$\bar{\Sigma}$). We start with the averaged system.

Let~$\text{\sffamily{\upshape{M}}}$ be a closed submanifold of a Euclidean space. For every vector~$w\in\mathbb{R}^n$ of disturbances, let~$\bar{\xi}\mapsto\bar{\Sigma}(\bar{\xi},w)$ be a vector field on~$\text{\sffamily{\upshape{M}}}$. It is assumed that, for every~$w\in{L_{\infty}^{n}}$, every~$t_0\in\mathbb{R}$, and every~$\xi_0\in\text{\sffamily{\upshape{M}}}$, the initial value problem
\begin{equation}\label{eq:01}
\dot{\bar{\xi}}(t) \ = \ \bar{\Sigma}(\bar{\xi}(t),w(t)), \qquad \bar{\xi}(t_0) \ = \ \xi_0
\end{equation}
has a unique maximal solution. We consider the following notions of stability for~$\bar{\Sigma}$; see, e.g.,~\cite{Sontag1995}.
\begin{defn}\label{def:01}
Let $K\subset\text{\sffamily{\upshape{M}}}$ be nonempty and compact. We say that~$\bar{\Sigma}$ is \emph{0-\textbf{g}lobally \textbf{a}symptotically \textbf{s}table (0-GAS)} w.r.t.~$K$ if there exists~$\bar{\beta}\in\mathcal{K}\mathcal{L}$ such that, for every~$\xi_0\in\text{\sffamily{\upshape{M}}}$, the maximal solution~$\bar{\xi}$ of~\cref{eq:01} with~$w\equiv0$ and~$t_0=0$ satisfies
\begin{equation*}
|\bar{\xi}(t)|_K \ \leq \ \bar{\beta}(|\xi_0|_K,t)\qquad\forall{t\geq0}.
\end{equation*}
\end{defn}
\begin{defn}\label{def:02}
Let $K\subset\text{\sffamily{\upshape{M}}}$ be nonempty and compact. We say that~$\bar{\Sigma}$ is \emph{\textbf{i}nput-to-\textbf{s}tate \textbf{s}table (ISS)} w.r.t.~$K$ if there exist~$\bar{\beta}\in\mathcal{K}\mathcal{L}$ and~$\bar{\gamma}\in\mathcal{K}$ such that, for every~$w\in{L_{\infty}^{n}}$ and every~$\xi_0\in\text{\sffamily{\upshape{M}}}$, the maximal solution~$\bar{\xi}$ of~\cref{eq:01} with~$t_0=0$ satisfies
\begin{equation*}
|\bar{\xi}(t)|_K \ \leq \ \bar{\beta}(|\xi_0|_K,t) + \bar{\gamma}(\|w\|)\qquad\forall{t\geq0}.
\end{equation*}
\end{defn}
For every~$\omega\in\mathbb{R}_+$ and every~$w\in\mathbb{R}^n$, let~$(t,\tilde{\xi})\mapsto\tilde{\Sigma}^{\omega}(t,\tilde{\xi},w)$ be a time-dependent vector field on~$\text{\sffamily{\upshape{M}}}$. It is assumed that, for every~$\omega\in\mathbb{R}_+$, every~$w\in{L_{\infty}^{n}}$, every~$t_0\in\mathbb{R}$, and every~$\xi_0\in\text{\sffamily{\upshape{M}}}$, the initial value problem
\begin{equation}\label{eq:02}
\dot{\tilde{\xi}}(t) \ = \ \tilde{\Sigma}^{\omega}(t,\tilde{\xi}(t),w(t)), \qquad \tilde{\xi}(t_0) \ = \ \xi_0
\end{equation}
has a unique maximal solution. We use the notation~$(\tilde{\Sigma}^{\omega})_{\omega}$ for the family of~$\tilde{\Sigma}^{\omega}$ indexed by $\omega\in\mathbb{R}_+$. In the subsequent sections, the parameter~$\omega$ will play the role of a frequency. The map~$\tilde{\Sigma}^{\omega}$ will be chosen in such a way that, with increasing value of~$\omega$, the solutions of~\cref{eq:02} approximate the solutions of~\cref{eq:01} on compact time intervals and within compact subsets of~$\text{\sffamily{\upshape{M}}}$. For our purposes, it turns out to be convenient to introduce the following two notions of local approximations.
\begin{defn}\label{def:03}
We say that~$(\tilde{\Sigma}^{\omega})_{\omega}$ is a \emph{small-disturbance approximation} of~$\bar{\Sigma}$ if, for every compact set~$C\subset\text{\sffamily{\upshape{M}}}$ and all~$\delta,\Delta\in\mathbb{R}_+$, there exist~$\omega_0,e\in\mathbb{R}_+$ such that, for every~$t_0\in\mathbb{R}$ and every~$\xi_0\in{C}$, the following implication holds: If the maximal solution~$\bar{\xi}$ of~\cref{eq:01} with~$w\equiv0$ satisfies~$\bar{\xi}(t)\in{C}$ for every~$t\in[t_0,t_0+\Delta]$, then, for every~$\omega\geq\omega_0$ and every~$w\in{L_{\infty}^n}$ with~$\|w\|\leq{e}$, the maximal solution~$\tilde{\xi}$ of~\cref{eq:02} satisfies
\begin{equation*}
|\tilde{\xi}(t) - \bar{\xi}(t)| \ \leq \ \delta\qquad\forall{t\in[t_0,t_0+\Delta]}.
\end{equation*}
\end{defn}
\begin{defn}\label{def:04}
We say that~$(\tilde{\Sigma}^{\omega})_{\omega}$ is a \emph{large-disturbance approximation} of~$\bar{\Sigma}$ if, for every compact set~$C\subset\text{\sffamily{\upshape{M}}}$ and all~$\delta,\Delta,e\in\mathbb{R}_+$, there exists~$\omega_0\in\mathbb{R}_+$ such that, for every~$w\in{L_{\infty}^n}$ with~$\|w\|\leq{e}$, every~$t_0\in\mathbb{R}$, and every~$\xi_0\in{C}$, the following implication holds: If the maximal solution~$\bar{\xi}$ of~\cref{eq:01} satisfies~$\bar{\xi}(t)\in{C}$ for every~$t\in[t_0,t_0+\Delta]$, then, for every~$\omega\geq\omega_0$, the maximal solution~$\tilde{\xi}$ of~\cref{eq:02} satisfies
\begin{equation*}
|\tilde{\xi}(t) - \bar{\xi}(t)| \ \leq \ \delta\qquad\forall{t\in[t_0,t_0+\Delta]}.
\end{equation*}
\end{defn}
Note that \Cref{def:03,def:04} reduce to the same approximation property as in~\cite{Moreau2000,Duerr2013} if no disturbances are present (i.e.~$w\equiv0$). Next, we extend the notions of semi-global practical stability for parameter-dependent systems in~\cite{Moreau2000,Duerr2013} to systems with disturbances.
\begin{defn}\label{def:05}
Let $K\subset\text{\sffamily{\upshape{M}}}$ be nonempty and compact. We say that~$(\tilde{\Sigma}^{\omega})_{\omega}$ is \emph{small-disturbance \textbf{s}emi-\textbf{g}lobally \textbf{p}ractically \textbf{u}niformally \textbf{a}symptotically \textbf{s}table (SGPUAS)} w.r.t.~$K$ if there exists~$\tilde{\beta}\in\mathcal{K}\mathcal{L}$ such that, for all~$\rho,\nu\in\mathbb{R}_+$, there exist~$\omega_0,e\in\mathbb{R}_+$ such that, for every~$\omega\geq\omega_0$, every~$w\in{L_{\infty}^n}$ with~$\|w\|\leq{e}$, every~$t_0\in\mathbb{R}$, and every~$\xi_0\in\text{\sffamily{\upshape{M}}}$ with~$|\xi_0|_K\leq\rho$, the maximal solution~$\tilde{\xi}$ of~\cref{eq:02} satisfies
\begin{equation*}
|\tilde{\xi}(t)|_K \ \leq \ \tilde{\beta}(|\xi_0|_K,t-t_0) + \nu\qquad\forall{t\geq{t_0}}.
\end{equation*}
\end{defn}
\begin{defn}\label{def:06}
Let $K\subset\text{\sffamily{\upshape{M}}}$ be nonempty and compact. We say that~$(\tilde{\Sigma}^{\omega})_{\omega}$ is \emph{large-disturbance SGPUAS} w.r.t.~$K$ if there exist~$\tilde{\beta}\in\mathcal{K}\mathcal{L}$ and~$\tilde{\gamma}\in\mathcal{K}$ such that, for all~$\rho,\nu,e\in\mathbb{R}_+$, there exists~$\omega_0\in\mathbb{R}_+$ such that, for every~$\omega\geq\omega_0$, every~$w\in{L_{\infty}^n}$ with~$\|w\|\leq{e}$, every~$t_0\in\mathbb{R}$, and every~$\xi_0\in\text{\sffamily{\upshape{M}}}$ with~$|\xi_0|_K\leq\rho$, the maximal solution~$\tilde{\xi}$ of~\cref{eq:02} satisfies
\begin{equation*}
|\tilde{\xi}(t)|_K \ \leq \ \tilde{\beta}(|\xi_0|_K,t-t_0) + \tilde{\gamma}(\|w\|) + \nu\qquad\forall{t\geq{t_0}}.
\end{equation*}
\end{defn}
The key difference between \Cref{def:05,def:06} is the maximum magnitude~$e\in\mathbb{R}_+$ of disturbances. While \Cref{def:05} requires that~$e\in\mathbb{R}_+$ is sufficiently small, \Cref{def:06} allows arbitrary large~$e\in\mathbb{R}_+$. In contrast to the global notions of stability in \Cref{def:01,def:02}, the term ``semi-global'' in \Cref{def:05,def:06} emphasizes the restriction to arbitrary large but compact sets. This is due the fact that the approximations in \Cref{def:03,def:04} are only local properties. A similar reasoning as in~\cite{Moreau2000,Duerr2013} leads to following result (we omit the proof).
\begin{prop}\label{prop:01}
Let~$K\subset\text{\sffamily{\upshape{M}}}$ be nonempty and compact. Then, the following implications hold:
\begin{enumerate}[label=(\alph*)]
	\item If~$\bar{\Sigma}$ is 0-GAS w.r.t.~$K$ and if~$(\tilde{\Sigma}^{\omega})_{\omega}$ is a small-disturbance approximation of~$\bar{\Sigma}$, then~$(\tilde{\Sigma}^{\omega})_{\omega}$ is small-disturbance SGPUAS w.r.t.~$K$.
	\item If~$\bar{\Sigma}$ is ISS w.r.t.~$K$ and if~$(\tilde{\Sigma}^{\omega})_{\omega}$ is a large-disturbance approximation of~$\bar{\Sigma}$, then~$(\tilde{\Sigma}^{\omega})_{\omega}$ is large-disturbance SGPUAS w.r.t.~$K$.
\end{enumerate}
\end{prop}


\section{Robust Lie Bracket Approximations}\label{sec:04}
We already know from \Cref{prop:01} that local approximations of trajectories (in the sense of \Cref{def:03,def:04}) lead to the effect that certain robustness and stability properties carry over from one system to another. The same principle holds for the extremum seeking system that we study later in \Cref{sec:05}. To be more precise, we will see that, after a suitable change of coordinates, the closed-loop system is a parameter- and time-dependent system, denoted by~$\tilde{\Sigma}^{a,\omega}$, that approximates the trajectories of an averaged system, denoted by~$\bar{\Sigma}^a$, with increasing parameter~$\omega\in\mathbb{R}_+$. In this section, we study the underlying approximation properties for a slightly more general type of system than the particular extremum seeking system in \Cref{sec:05}. As in~\cite{Duerr2013}, our investigations will lead us to a differential geometric explanation in terms of Lie brackets. In contrast to the approach in~\cite{Duerr2013}, we approximate Lie brackets in such a way that disturbances are not amplified by the oscillatory dither signals; see also \Cref{rem:05} in \Cref{sec:05}. This in turn leads to a certain degree of robustness with respect to general~$L_{\infty}$-disturbances.

Throughout this section, we suppose that
\begin{itemize}
	\item~$\text{\sffamily{\upshape{M}}}$ is a closed submanifold of a Euclidean space,
	\item~$X_1,\ldots,X_l\in\mathfrak{X}(\text{\sffamily{\upshape{M}}})$, $l\in\mathbb{N}$,
	\item~$Y_0,Y_1,\ldots,Y_m\in\mathfrak{X}(\text{\sffamily{\upshape{M}}})$, $m\in\mathbb{N}$,
	\item~$Z_1,\ldots,Z_n\in\mathfrak{X}(\text{\sffamily{\upshape{M}}})$, $n\in\mathbb{N}$,
	\item~$u\in{L_{\infty}^l}$, $v\in{L_{\infty}^m}$.
\end{itemize}
The maps~$u,v$ shall play the role of \emph{oscillatory dither signals}. For this reason, we assume the following.
\begin{assum}\label{ass:01}
There exists~$T\in\mathbb{R}_+$ such that~$u,v$ are~$T$-periodic and zero-mean.
\end{assum}
\Cref{ass:01} causes the oscillatory dither signals~$u,v$ to resonate. This in turn can lead to an approximation of Lie brackets of the~$X_i$ and the~$Y_j$ if we combine the dither signals and the vector fields as follows. Define a time-dependent vector field~$\hat{X}$ on~$\text{\sffamily{\upshape{M}}}$ by
\begin{equation*}
\hat{X}(\tau,\xi) \ := \ u^i(\tau)\,X_i(\xi)
\end{equation*}
in the \hyperlink{sumConvention}{summation convention} of \Cref{sec:02}. For every~$a\in\mathbb{R}_+$, define a time-dependent vector field~$\hat{Y}^a$ on~$\text{\sffamily{\upshape{M}}}$ by
\begin{equation*}
\hat{Y}^a(\tau,\xi) \ := \ Y_0(\xi) + \frac{1}{a}\,v^j(\tau)\,Y_j(\xi).
\end{equation*}
For every vector~$w\in\mathbb{R}^n$ of \emph{disturbances}, define a smooth vector field~$\xi\mapsto\hat{Z}(\xi,w)$ on~$\text{\sffamily{\upshape{M}}}$ by
\begin{equation*}
\hat{Z}(\xi,w) \ := \ w^k\,Z_k(\xi).
\end{equation*}
For all~$a,\omega\in\mathbb{R}_+$ and every~$w\in\mathbb{R}^n$, define a time-dependent vector field~$(t,\xi)\mapsto\Sigma^{a,\omega}(t,\xi,w)$ on~$\text{\sffamily{\upshape{M}}}$ by
\begin{equation}\label{eq:03}
\Sigma^{a,\omega}(t,\xi,w) \ := \ a\omega\hat{X}(\omega{t},\xi) + \hat{Y}^a(\omega{t},\xi) + \hat{Z}(\xi,w).
\end{equation}
In \Cref{sec:05}, the map~$\Sigma^{a,\omega}$ describes the right-hand side of the closed-loop system. In what follows, we study the integral curves of~$\Sigma^{a,\omega}$, for fixed~$a\in\mathbb{R}_+$, in the large-amplitude, high-frequency limit~$\omega\to\infty$. To get rid of the large-amplitude, high-frequency term in~\cref{eq:03}, we take the pull-back of~$\hat{Y}^a$ and~$\hat{Z}$ by the flow of~$\hat{X}$. In general, we cannot expect that~$\hat{X}$ is complete. However, the situation changes if we assume the following.
\begin{assum}\label{ass:02}
The vector fields~$X_1,\ldots,X_l$ are complete and commute pairwise.
\end{assum}
The above assumption ensures that the integral curves of~$\hat{X}$ are~$T$-periodic. To make this statement more precise, we define~$X^{\tau}\in\mathfrak{X}(\text{\sffamily{\upshape{M}}})$ for every fixed~$\tau\in\mathbb{R}$ by
\begin{equation}\label{eq:04}
X^{\tau}(\xi) \ := \ U^i(\tau)\,X_i(\xi),
\end{equation}
where~$U\colon\mathbb{R}\to\mathbb{R}^l$ is an antiderivative of~$u$ defined by
\begin{equation}\label{eq:05}
U(\tau) \ := \ \int_{0}^{\tau}u(\sigma)\,\mathrm{d}\sigma.
\end{equation}
Now, the flow of the large-amplitude, high-frequency vector field in~\cref{eq:03} is given by the flow of~$X^{\tau}$ as follows (see, e.g., Proposition~9.13 and Remark~9.14 in~\cite{BulloBook}).
\begin{rem}\label{rem:01}
Suppose that \Cref{ass:01,ass:02} are satisfied. Then, for every~$\tau\in\mathbb{R}$, the vector field~$X^{\tau}$ is complete. Moreover, for all~$a,\omega\in\mathbb{R}_+$ and every~$\xi_0\in\text{\sffamily{\upshape{M}}}$, the maximal solution~$\xi\colon{I}\to\text{\sffamily{\upshape{M}}}$ of
\begin{equation*}
\dot{\xi}(t) \ = \ a\omega\hat{X}(\omega{t},\xi(t)), \qquad \xi(0) \ = \ \xi_0
\end{equation*}
is given by
\begin{equation*}
\xi(t) \ = \ \Phi^{X^{\omega{t}}}_a(\xi_0)
\end{equation*}
for every~$t\in{I}=\mathbb{R}$.
\end{rem}
For the sake of simplicity, we also make the following assumption (which is trivially satisfied for the particular problem studied in \Cref{sec:05}).
\begin{assum}\label{ass:03}
For every~$i\in\{1,\ldots,l\}$ and every~$k\in\{1,\ldots,n\}$, the vector fields~$X_i$ and~$Z_k$ commute.
\end{assum}
Because of \Cref{ass:03}, we may conclude the following simplifying identities from the well-known \emph{flow interpretation of the Lie derivative of a vector field} (see, e.g., Proposition~3.85 in~\cite{BulloBook}).
\begin{rem}\label{rem:02}
Suppose that \Cref{ass:01,ass:02,ass:03} are satisfied. Then~$(\Phi^{X^{\tau}}_a)^{\ast}X_i=X_i$ and~$(\Phi^{X^{\tau}}_a)^{\ast}Z_k=Z_k$ for every~$a\in\mathbb{R}_+$, every~$\tau\in\mathbb{R}$, and all~$i\in\{1,\ldots,l\}$, $k\in\{1,\ldots,n\}$.
\end{rem}
For the rest of this section, we suppose that \Cref{ass:01,ass:02,ass:03} are satisfied. Because of \Cref{rem:01}, for every~$a\in\mathbb{R}_+$, a well-defined time-dependent vector field~$\tilde{Y}^a$ on~$\text{\sffamily{\upshape{M}}}$ is given by
\begin{equation*}
\tilde{Y}^a(\tau,\tilde{\xi}) \ := \ ((\Phi^{X^{\tau}}_a)^{\ast}\hat{Y}^a_{\tau})(\tilde{\xi}).
\end{equation*}
For all~$a,\omega\in\mathbb{R}_+$ and every~$w\in\mathbb{R}^n$, define a time-dependent vector field~$(t,\tilde{\xi})\mapsto\tilde{\Sigma}^{a,\omega}(t,\tilde{\xi},w)$ on~$\text{\sffamily{\upshape{M}}}$ by
\begin{equation}\label{eq:06}
\tilde{\Sigma}^{a,\omega}(t,\tilde{\xi},w) \ := \ \tilde{Y}^a(\omega{t},\tilde{\xi}) + \hat{Z}(\tilde{\xi},w).
\end{equation}
In \Cref{sec:05}, the map~$\tilde{\Sigma}^{a,\omega}$ describes the right-hand side of the closed-loop system after the change of coordinates in \Cref{rem:03} below.

Because of \Cref{rem:01,rem:02}, we may say that~$\tilde{\Sigma}^{a,\omega}$ is the pull-back of~$\Sigma^{a,\omega}$ by~$\Phi^{X^{\omega{t}}}_a$. The \emph{variation of constants formula} (see, e.g., Proposition~9.6 in~\cite{BulloBook}) provides the following connection between the integral curves.
\begin{rem}\label{rem:03}
Suppose that \Cref{ass:01,ass:02,ass:03} are satisfied. Then, for all~$a,\omega\in\mathbb{R}_+$, every~$w\in{L_{\infty}^n}$, every~$t_0\in\mathbb{R}$, and every~$\tilde{\xi}_0\in\text{\sffamily{\upshape{M}}}$, the maximal solution~$\xi\colon{I}\to\text{\sffamily{\upshape{M}}}$ of
\begin{equation*}
\dot{\xi}(t) \ = \ \Sigma^{a,\omega}(t,\xi(t),w(t)), \qquad \xi(t_0) \ = \ \Phi^{X^{\omega{t_0}}}_a(\tilde{\xi}_0)
\end{equation*}
and the maximal solution~$\tilde{\xi}\colon\tilde{I}\to\text{\sffamily{\upshape{M}}}$ of
\begin{equation*}
\dot{\tilde{\xi}}(t) \ = \ \tilde{\Sigma}^{a,\omega}(t,\tilde{\xi}(t),w(t)), \qquad \tilde{\xi}(t_0) \ = \ \tilde{\xi}_0
\end{equation*}
are related by the change of coordinates
\begin{equation}\label{eq:07}
\xi(t) \ = \ \Phi^{X^{\omega{t}}}_a(\tilde{\xi}(t))
\end{equation}
for every~$t\in{I=\tilde{I}}$.
\end{rem}
Note that~$\tilde{Y}^{a}$ is~$T$-periodic. For every~$a\in\mathbb{R}_+$, define the \emph{averaged vector field}~$\bar{Y}^a\in\mathfrak{X}(\text{\sffamily{\upshape{M}}})$ of~$\tilde{Y}^a$ by
\begin{equation*}
\bar{Y}^a(\bar{\xi}) \ := \ \frac{1}{T}\int_{0}^{T}\tilde{Y}^a(\tau,\bar{\xi})\,\mathrm{d}\tau
\end{equation*}
(see, e.g., Section~9.1 in~\cite{BulloBook}). For every~$a\in\mathbb{R}_+$ and every~$w\in\mathbb{R}^n$, define a vector field~$\bar{\xi}\mapsto\bar{\Sigma}^a(\bar{\xi},w)$ on~$\text{\sffamily{\upshape{M}}}$ by
\begin{equation}\label{eq:08}
\bar{\Sigma}^a(\bar{\xi},w) \ := \ \bar{Y}^a(\bar{\xi}) + \hat{Z}(\bar{\xi},w).
\end{equation}
In \Cref{sec:05}, the map~$\bar{\Sigma}^{a}$ describes the right-hand side of the averaged system associated with closed-loop system.

By applying a suitable \emph{first-order averaging} procedure (see, e.g., proof of Theorem~9.15 in~\cite{BulloBook}) and the Gronwall lemma, one can prove the following approximation properties (in the terminology of \Cref{def:03,def:04}).
\begin{prop}\label{prop:02}
Suppose that \Cref{ass:01,ass:02,ass:03} are satisfied. Fix an arbitrary~$a\in\mathbb{R}_+$. Then:
\begin{enumerate}[label=(\alph*)]
	\item~$\!\!(\tilde{\Sigma}^{a,\omega})_{\omega}$ is a small-disturbance approximation of~$\bar{\Sigma}^a$,
	\item~$\!\!(\tilde{\Sigma}^{a,\omega})_{\omega}$ is a large-disturbance approximation of~$\bar{\Sigma}^a$.
\end{enumerate}
\end{prop}
Finally, we provide a more explicit formula for the averaged vector field~$\bar{Y}^a$. To state this formula, we define the iterated integral~$\overline{Uv}^{i,j}\in\mathbb{R}$ by
\begin{equation}\label{eq:09}
\overline{Uv}^{i,j} \ := \ \frac{1}{T}\int_{0}^{T}U^i(\tau)\,v^j(\tau)\mathrm{d}\tau
\end{equation}
for every~$i\in\{1,\ldots,l\}$ and every~$j\in\{1,\ldots,m\}$.
\begin{exmp}\label{exm:01}
Let~$\alpha,c\in\mathbb{R}_+^m$. Let~$\varpi\in\mathbb{N}^m$ have pairwise distinct entries. Let~$n=m$ and~$T=2\pi$. Define~$u,v\colon\mathbb{R}\to\mathbb{R}^m$ component-wise by
\begin{equation*}
u^i(\tau) \ := \ \alpha^i\,\varpi^i\,\cos(\varpi^i\tau), \qquad v^j(\tau) \ := \ c^j\,\sin(\varpi^j\tau).
\end{equation*}
Then, \Cref{ass:01} is satisfied and~\cref{eq:09} is given by
\begin{equation}\label{eq:10}
\overline{Uv}^{i,j} \ = \ \alpha^i\,c^j\,\delta^{i,j}/2
\end{equation}
for all~$i,j\in\{1,\ldots,m\}$, where~$\delta^{i,j}$ denotes the Kronecker delta of~$i$ and~$j$.
\end{exmp}
An expansion of~$(\Phi^{X^{\tau}}_a)^{\ast}$ around~$a=0$ (using Propo\-sition~3.85 in~\cite{BulloBook}) leads to the following formula for~$\bar{Y}^a$.
\begin{rem}\label{rem:04}
Suppose that \Cref{ass:01,ass:02,ass:03} are satisfied. Suppose that~$U$ is zero-mean. Then
\begin{equation}\label{eq:11}
\bar{Y}^a(\bar{\xi}) \ = \ \bar{Y}^0(\bar{\xi}) + \delta\bar{Y}^a(\bar{\xi})
\end{equation}
for every~$a\in\mathbb{R}_+$ and every~$\bar{\xi}\in\text{\sffamily{\upshape{M}}}$, where\addtocounter{equation}{+1}\refstepcounter{equation}\label{eq:13}\addtocounter{equation}{-2}
\begin{align}
\bar{Y}^0(\bar{\xi}) & \ := \ Y_0(\bar{\xi}) + \overline{Uv}^{i,j}[X_i,Y_j](\bar{\xi}), \label{eq:12}\allowdisplaybreaks \\
\delta\bar{Y}^a(\bar{\xi}) & \ := \ \frac{a}{T}\int_0^T\int_0^1(1-s)\,U^{i_1}(\tau)\,U^{i_2}(\tau) \stepcounter{equation}\nonumber\tag{\arabic{equation}a}\allowdisplaybreaks \\
& \!\!\!\!\!\!\!\!\!\!\times\Big(a\,((\Phi^{X^{\tau}}_{sa})^{\ast}[X_{i_1}[X_{i_2},Y_0]])(\bar{\xi}) \nonumber\tag{\arabic{equation}b}\allowdisplaybreaks \\
&  + v^j(\tau)\,((\Phi^{X^{\tau}}_{sa})^{\ast}[X_{i_1}[X_{i_2},Y_j]])(\bar{\xi})\Big)\mathrm{d}s\,\mathrm{d}\tau. \nonumber\tag{\arabic{equation}c}
\end{align}
Note that the \emph{main part}~$\bar{Y}^0$ of~$\bar{Y}^a$ in~\cref{eq:12} is the same as the averaged vector field in~\cite{Duerr2013}. The \emph{remainder vector field}~$\delta\bar{Y}^a$ in~\cref{eq:13} vanishes as~$a\to0$.
\end{rem}


\section{Extremum Seeking Control}\label{sec:05}
Throughout this section, we suppose that
\begin{itemize}
	\item~$\text{\sffamily{\upshape{X}}}$ is a closed submanifold of a Euclidean space,
	\item~$F_0,F_1,\ldots,F_m\in\mathfrak{X}(\text{\sffamily{\upshape{X}}})$, $m\in\mathbb{N}$,
	\item~$\psi\in\mathfrak{F}(\text{\sffamily{\upshape{X}}})$.
\end{itemize}
In the \hyperlink{sumConvention}{summation convention} of \Cref{sec:02}, we consider a multiple-input single-output system on~$\text{\sffamily{\upshape{X}}}$ of the form
\begin{align}
\dot{x}(t) & \ = \ F_0(x(t)) + \mathrm{u}^i\,F_i(x(t)), \label{eq:14} \\
\mathrm{y} & \ = \ \psi(x(t)), \nonumber
\end{align}
where~$\mathrm{u}$ is an~$m$-component vector of real-valued input channels~$\mathrm{u}^1,\ldots,\mathrm{u}^m$ and~$\mathrm{y}$ is a real-valued output channel. We are interested in an output-feedback law that stabilizes the closed-loop system around states where~$\psi$ attains a extreme value. Since our approach is closely related to~\cite{Zhang20071,Zhang20072}, we follow the convention therein and focus on \emph{maxima} of~$\psi$. It is assumed that, at any time~$t\in\mathbb{R}$, a measurement of~$\mathrm{y}$ results in a noise-corrupted value
\begin{equation*}
\hat{y}(t) \ = \ \psi(x(t)) + d_{\mathrm{y}}(t),
\end{equation*}
where~$d_{\mathrm{y}}\in{L_{\infty}}$ is an unknown output disturbance.

As in \Cref{sec:04}, we choose oscillatory dither signals~$u,v\in{L_{\infty}^m}$ such that \Cref{ass:01} is satisfied. For instance, we can use the sinusoids in \Cref{exm:01}. We consider the time-dependent output-feedback control law
\begin{equation}\label{eq:15}
\mathrm{u} \ = \ a\,\omega\,u(\omega{t}) + \frac{1}{a}\,v(\omega{t})\,\big(\hat{y}(t)-\eta(t)\big) + d_{\mathrm{u}}(t)
\end{equation}
with control parameters~$a,\omega\in\mathbb{R}_+$, where~$d_{\mathrm{u}}\in{L_{\infty}^m}$ is a vector of input disturbances and~$\eta$ is the real-valued state variable of a high-pass filter
\begin{equation}\label{eq:16}
\dot{\eta}(t) \ = \ -h\,\eta(t) + h\,\hat{y}(t)
\end{equation}
with gain~$h\in\mathbb{R}_+$ to remove a possible offset from~$\hat{y}(t)$. The control scheme is depicted in \Cref{fig:ESCScheme}.
\begin{figure}%
\centering\includegraphics{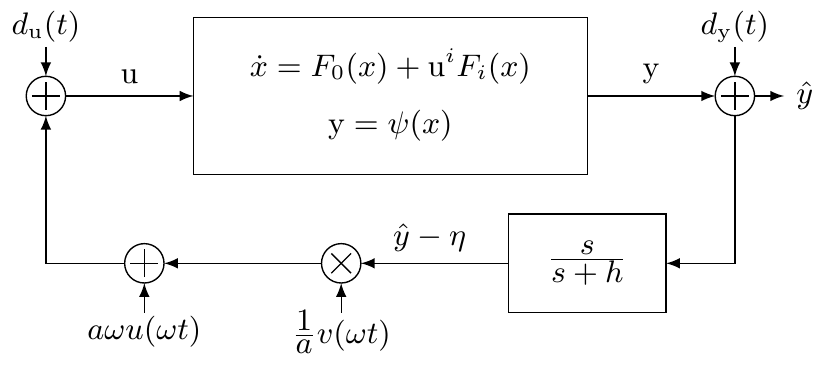}%
\caption{Sketch of a large-amplitude, high-frequency extremum seeking control scheme. The closed-loop system is described by equation~\cref{eq:20}.}%
\label{fig:ESCScheme}%
\end{figure}
\begin{rem}\label{rem:05}
Control law \Cref{eq:15} is studied in~\cite{Zhang20071,Zhang20072} for~$a=1$ and in~\cite{Duerr2013} for~$a=1/\sqrt{\omega}$, but, in any case, without taking disturbances into account. The choice of the parameter~$a$ is decisive for robustness in the large-amplitude, high-frequency limit~$\omega\to\infty$. On the one hand, a sufficiently large value of~$\omega$ is necessary to guarantee a good approximation of the averaged system. On the other hand, if~$a=1/\sqrt{\omega}$ as in~\cite{Duerr2013}, then the noise-corrupted signal~$\hat{y}-\eta$ is amplified by the dither signal~$t\mapsto\sqrt{\omega}\,v(\omega{t})$, which can lead to a complete loss of stability when~$\omega$ is large (see, e.g., \Cref{fig:unicycle}~\hyperlink{unicycleBR}{(d)} in \Cref{sec:06}). There is, however, no amplification of disturbances if~$a=1$ as in~\cite{Zhang20071,Zhang20072}. This very simple but crucial difference in the choice of~$a$ allows us to prove robustness of the closed-loop system for fixed~$a\in\mathbb{R}_+$ in the large-amplitude, high-frequency limit~$\omega\to\infty$. The additive dither signal~$t\mapsto{a\omega{u(\omega{t})}}$ certainly leads to strong oscillations of the system state~$x$, but the amplitudes are uniformly bounded with respect to~$\omega$.
\end{rem}
When we insert~\cref{eq:15} into~\cref{eq:14}, then we get the closed-loop system\refstepcounter{equation}\label{eq:17}
\begin{align}
& \dot{x}(t) \ = \ F_0(x(t)) + \big(a\omega\,u^i(\omega{t})+ d_{\mathrm{u}}^i(t)\big)\,F_i(x(t)) \nonumber\tag{\arabic{equation}a} \\
& \quad + \frac{1}{a}\,v^j(\omega{t})\,\big(\psi(x(t))+d_{\mathrm{y}}(t)-\eta(t)\big)\,F_j(x(t)), \nonumber\tag{\arabic{equation}b} \\
& \dot{\eta}(t) \ = \ -h\,\eta(t) + h\,\psi(x(t)) + h\,d_{\mathrm{y}}(t) \label{eq:18}
\end{align}
on the product manifold~$\text{\sffamily{\upshape{M}}}:=\text{\sffamily{\upshape{X}}}\times\mathbb{R}$. Now we are in a particular situation of \Cref{sec:04}. To make this apparent, we define the positive integers~$l:=m$, $n:=m+1$, and smooth vector fields~$X_i$, $Y_j$, $Z_k$ on~$\text{\sffamily{\upshape{M}}}$ by
\begin{align*}
X_i(\xi) & \ := \ Z_i(\xi) \ := \ \begin{bmatrix} F_i(x) \\ 0 \end{bmatrix}, \qquad i\in\{1,\ldots,m\}, \allowdisplaybreaks \\
Y_i(\xi) & \ := \ (\psi(x)-\eta)X_i(\xi), \qquad i\in\{1,\ldots,m\}, \allowdisplaybreaks \\
Y_0(\xi) & \ := \ \begin{bmatrix} F_0(x) \\ -h\,\eta + h\,\psi(x) \end{bmatrix}, \qquad Z_n(\xi) \ := \ \begin{bmatrix} 0 \\ h \end{bmatrix}
\end{align*}
for every~$\xi=[x,\eta]\in\text{\sffamily{\upshape{M}}}$.
\begin{assum}\label{ass:04}
The vector fields~$F_1,\ldots,F_m$ are complete and commute pairwise.
\end{assum}
Suppose that \Cref{ass:01,ass:04} are satisfied. Then, it is easy to check that also \Cref{ass:02,ass:03} are satisfied and therefore all definitions and statements in \Cref{sec:04} apply to the specific problem in this section; in particular, the definitions of~$\Sigma^{a,\omega}$ and~$\tilde{\Sigma}^{a,\omega}$ in~\cref{eq:03,eq:08}, respectively. For all~$a,\omega\in\mathbb{R}_+$ and every~$d=[d_{\mathrm{u}}^{\top},d_{\mathrm{y}}]^{\top}$ with~$d_{\mathrm{u}}\in\mathbb{R}^m$ and~$d_{\mathrm{y}}\in\mathbb{R}$, define time-dependent vector fields~$(t,\xi)\mapsto\Sigma^{a,\omega}_{\text{ES}}(t,\xi,d)$, $(t,\tilde{\xi})\mapsto\tilde{\Sigma}^{a,\omega}_{\text{ES}}(t,\tilde{\xi},d)$ on~$\text{\sffamily{\upshape{M}}}$ by
\begin{align*}
\Sigma^{a,\omega}_{\text{ES}}(t,\xi,d) & \ := \ \Sigma^{a,\omega}(t,\xi,w_{a,\omega{t},d}), \\
\tilde{\Sigma}^{a,\omega}_{\text{ES}}(t,\tilde{\xi},d) & \ := \ \tilde{\Sigma}^{a,\omega}(t,\tilde{\xi},w_{a,\omega{t},d}),
\end{align*}
where the components of~$w_{a,\omega{t},d}\in\mathbb{R}^n$ are given by
\begin{subequations}
\begin{align}
w_{a,\omega{t},d}^i & \ = \ d_{\mathrm{u}}^i + v^i(\omega{t})\,d_{\mathrm{y}}/a, \qquad i=1,\ldots,m, \label{eq:19A} \\
w_{a,\omega{t},d}^{n} & \ = \ d_{\mathrm{y}},
\end{align}
\end{subequations}
and the index ES stands for \textbf{e}xtremum \textbf{s}eeking.
\begin{rem}\label{rem:06}
For all~$a,\omega\in\mathbb{R}_+$ and every~$d=[d_{\mathrm{u}}^{\top},d_{\mathrm{y}}]^{\top}$ with~$d_{\mathrm{u}}\in{L_{\infty}^m}$ and~$d_{\mathrm{y}}\in{L_{\infty}}$, the system
\begin{equation}\label{eq:20}
\dot{\xi}(t) \ = \ \Sigma^{a,\omega}_{\text{ES}}(t,\xi(t),d(t))
\end{equation}
on~$\text{\sffamily{\upshape{M}}}$ coincides with the closed-loop system~\cref{eq:17},~\cref{eq:18}. Since~$\tilde{\Sigma}_{\text{ES}}^{a,\omega}$ is the pull-back of~$\Sigma_{\text{ES}}^{a,\omega}$ by~$\Phi^{X^{\omega{t}}}_a$, we call
\begin{equation}\label{eq:21}
\dot{\tilde{\xi}}(t) \ = \ \tilde{\Sigma}_{\text{ES}}^{a,\omega}(t,\tilde{\xi}(t),d(t))
\end{equation}
the \emph{pull-back system of~\cref{eq:20}}.
\end{rem}
For every~$\tau\in\mathbb{R}$, define~$F^{\tau}\in\mathfrak{X}(\text{\sffamily{\upshape{X}}})$ by
\begin{equation}\label{eq:22}
F^{\tau}(x) \ := \ U^i(\tau)\,F_i(x),
\end{equation}
where~$U\colon\mathbb{R}\to\mathbb{R}^m$ is defined by~\cref{eq:05}. Note that, if \Cref{ass:01,ass:04} are satisfied, then~$F^{\tau}$ is complete (cf. \Cref{rem:01}). As a direct consequence of \Cref{rem:03}, we get the subsequent connection between the solutions of~\cref{eq:20} and~\cref{eq:21}.
\begin{rem}\label{rem:07}
Suppose that \Cref{ass:01,ass:04} are satisfied. Then, for all~$a,\omega\in\mathbb{R}_+$, every~$d\in{L_{\infty}^n}$, every~$t_0\in\mathbb{R}$, and every~$\tilde{\xi}_0=[\tilde{x}_0,\tilde{\eta}_0]\in\text{\sffamily{\upshape{M}}}$, the maximal solution~$\xi=[x,\eta]\colon{I}\to\text{\sffamily{\upshape{M}}}$ of~\cref{eq:20} with initial condition~$\xi(t_0)=[\Phi^{F^{\omega{t_0}}}_a(\tilde{x}_0),\tilde{\eta}_0]$ and the maximal solution~$\tilde{\xi}=[\tilde{x},\tilde{\eta}]\colon\tilde{I}\to\text{\sffamily{\upshape{M}}}$ of~\cref{eq:21} with initial condition~$\tilde{\xi}(t_0)=\tilde{\xi}_0$ are related by the change of coordinates
\begin{equation}\label{eq:23}
x(t) \ = \ \Phi^{F^{\omega{t}}}_a(\tilde{x}(t)), \qquad \eta(t) \ = \ \tilde{\eta}(t)
\end{equation}
for every~$t\in{I=\tilde{I}}$.
\end{rem}
In the next step, we provide a formula for the averaged vector field~$\bar{\xi}\mapsto\bar{\Sigma}^a(\bar{\xi},w)$ in~\cref{eq:08} under the following additional assumption.
\begin{assum}\label{ass:05}
$U$ is zero-mean.
\end{assum}
For the rest of this section, we suppose that \Cref{ass:01,ass:04,ass:05} are satisfied. We already know from \Cref{rem:04} how to compute the averaged vector field~$\bar{Y}^{a}$. By~\cref{eq:11}, it is the sum of the main part~$\bar{Y}^0$ in~\cref{eq:12} and the remainder vector field~$\delta\bar{Y}^a$ in~\cref{eq:13}. Note that, because of \Cref{ass:04}, we have
\begin{equation}\label{eq:24}
[F_i,\psi{F_j}](\bar{x}) \ = \ (F_i\psi)(\bar{x})\,F_j(\bar{x})
\end{equation}
for all~$i,j\in\{1,\ldots,m\}$ and every~$\bar{x}\in\text{\sffamily{\upshape{X}}}$, where~$F_{i}\psi$ denotes the Lie derivative of~$\psi$ along~$F_i$. Using~\cref{eq:24}, a direct computation reveals that~\cref{eq:12} and~\cref{eq:13} reduce to
\begin{equation*}
\bar{Y}^0(\bar{\xi}) \ = \ \begin{bmatrix} \bar{G}^0(\bar{x}) \\ -h\,\bar{\eta}+h\,\psi(\bar{x}) \end{bmatrix} \quad \text{and} \quad \delta\bar{Y}^a(\bar{\xi}) \ = \ \begin{bmatrix} \delta\bar{G}^a(\bar{x}) \\ \delta\bar{g}^a(\bar{x}) \end{bmatrix}
\end{equation*}
for every~$a\in\mathbb{R}_+$ and every~$\bar{\xi}=[\bar{x},\bar{\eta}]\in\text{\sffamily{\upshape{M}}}$, where
\begin{equation}\label{eq:25}
\bar{G}^0(\bar{x}) \ := \ F_0(\bar{x}) + \overline{Uv}^{i,j}\,(F_{i}\psi)(\bar{x})\,F_j(\bar{x})
\end{equation}
with coefficients~$\overline{Uv}^{i,j}\in\mathbb{R}$ as in~\cref{eq:09},
\begin{subequations}\label{eq:26}
\begin{align}
& \delta\bar{G}^{a}(\bar{x}) \ := \ \frac{a}{T}\int_0^T\int_0^1(1-s)\,U^{i_1}(\tau)\,U^{i_2}(\tau \allowdisplaybreaks) \\
& \quad\times\Big(a\big((\Phi^{F^{\tau}}_{sa})^{\ast}[F_{i_1},[F_{i_2},F_0]]\big)(\bar{x}) \allowdisplaybreaks \\
& \quad+v^j(\tau)\,(F_{i_1}(F_{i_2}\psi))(\Phi^{F^{\tau}}_{sa}(\bar{x}))\,F_j(\bar{x})\Big)\,\mathrm{d}s\,\mathrm{d}\tau
\end{align}
\end{subequations}
with~$F^{\tau}$ as in~\cref{eq:22}, and
\begin{align*}
\delta\bar{g}^{a}(\bar{x}) & \ := \ h\,\frac{a}{T}\int_0^T\int_0^1(1-s)\,U^{i_1}(\tau)\,U^{i_2}(\tau) \\
& \qquad \times{a}\,(F_{i_1}(F_{i_2}\psi))(\Phi^{F^{\tau}}_{sa}(\bar{x}))\,\mathrm{d}s\,\mathrm{d}\tau.
\end{align*}
\begin{rem}\label{rem:08}
The tangent vector~$\bar{G}^0(\bar{x})$ in~\cref{eq:25} contains valuable information about ascent directions of~$\psi$ at~$\bar{x}$. To see this, suppose that the oscillatory signals~$u,v$ are chosen as in \Cref{exm:01}. Then, we obtain from~\cref{eq:10} that the right-hand side of~\cref{eq:25} contains a linear combination of vectors of the form
\begin{equation}\label{eq:27}
[F_i,\psi{F_i}](\bar{x}) \ = \ (F_i\psi)(\bar{x})\,F_i(\bar{x})
\end{equation}
with positive linear coefficients. Note that the vector in~\cref{eq:27} points into an ascent direction of~$\psi$ at~$\bar{x}$ if~$(F_i\psi)(\bar{x})\neq0$. The purpose of the extremum scheme in \Cref{fig:ESCScheme} is to steer the control system into these potential ascent directions; i.e., towards a maximum of~$\psi$.
\end{rem}
Next, we insert the above formulas for~$\bar{Y}^0$ and~$\delta\bar{Y}^a$ into~\cref{eq:08} and~\cref{eq:11} to obtain~$\bar{\Sigma}^a$. For every~$a\in\mathbb{R}_+$, every~$w=[\bar{w}^{\top},w^n]^{\top}\in\mathbb{R}^n$ with~$\bar{w}\in\mathbb{R}^m$, $w^n\in\mathbb{R}$, and every~$\bar{\xi}=[\bar{x},\bar{\eta}]\in\text{\sffamily{\upshape{M}}}$, we get
\begin{equation}\label{eq:28}
\bar{\Sigma}^a(\bar{\xi},w) \ = \ \begin{bmatrix} \bar{\Sigma}^a_{\text{\sffamily{\upshape{X}}}}(\bar{x},\bar{w}) \\ \bar{\Sigma}^a_{\mathbb{R}}(\bar{\eta},\bar{x},w^n) \end{bmatrix},
\end{equation}
where the components~$\bar{\Sigma}^a_{\text{\sffamily{\upshape{X}}}}$ and~$\bar{\Sigma}^a_{\mathbb{R}}$ are given by
\begin{align}
\bar{\Sigma}^a_{\text{\sffamily{\upshape{X}}}}(\bar{x},\bar{w}) & \ := \ \bar{G}^0(\bar{x}) + \delta\bar{G}^a(\bar{x}) + \bar{w}^k\,F_k(\bar{x}), \label{eq:29} \\
\bar{\Sigma}^a_{\mathbb{R}}(\bar{\eta},\bar{x},w^n) & \ := -h\,\bar{\eta} + h\,\psi(\bar{x}) + \delta\bar{g}^a(\bar{x}) + h\,w^n. \label{eq:30}
\end{align}
Note that the dependence of~$\bar{\Sigma}^a_{\text{\sffamily{\upshape{X}}}}$ on the high-pass filter state~$\bar{\eta}$ drops out in the averaging process. Therefore, it is justified to refer to the system
\begin{equation}\label{eq:31}
\dot{\bar{x}}(t) \ = \ \bar{\Sigma}^a_{\text{\sffamily{\upshape{X}}}}(\bar{x}(t),\bar{w}(t))
\end{equation}
on~$\text{\sffamily{\upshape{X}}}$ as the \emph{averaged system of~\cref{eq:14} under~\cref{eq:15}}. Moreover, the independence of~$\bar{\Sigma}^a_{\text{\sffamily{\upshape{X}}}}$ on~$\bar{\eta}$ gives~\cref{eq:28} a cascade structure as indicated in \Cref{fig:cascade}.
\begin{figure}%
\centering\includegraphics{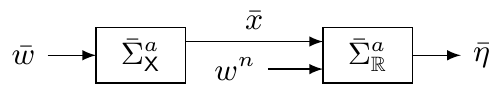}%
\caption{Representation of~$\bar{\Sigma}^a$ in~\cref{eq:28} as a cascade of~$\bar{\Sigma}^a_{\text{\sffamily{\upshape{X}}}}$ in~\cref{eq:29} and~$\bar{\Sigma}^a_{\mathbb{R}}$ in~\cref{eq:30}.}%
\label{fig:cascade}%
\end{figure}
Note that~$\dot{\bar{\eta}}=\bar{\Sigma}^a_{\mathbb{R}}(\bar{\eta},\bar{x},w^n)$ is ISS with~$(\bar{x},w^n)$ as inputs. Standard ISS results for cascade systems (see, e.g., Proposition~3.2 in~\cite{Jiang1994} and Remark~2.7 in~\cite{Sontag1995}) provide the following implications.
\begin{prop}\label{prop:03}
Suppose that \Cref{ass:01,ass:04,ass:05} are satisfied. Fix~$a\in\mathbb{R}_+$ and let~$K\subset\text{\sffamily{\upshape{X}}}$ be nonempty and compact. Then, the following implications hold:
\begin{enumerate}[label=(\alph*)]
	\item If~$\bar{\Sigma}^a_{\text{\sffamily{\upshape{X}}}}$ is 0-GAS w.r.t.~$K$, then there exists some nonempty, compact~$\Lambda\subset\mathbb{R}$ such that~$\bar{\Sigma}^a$ is 0-GAS w.r.t.~$K\times\Lambda$.
	\item If~$\bar{\Sigma}^a_{\text{\sffamily{\upshape{X}}}}$ is ISS w.r.t.~$K$, then there exists some nonempty, compact~$\Lambda\subset\mathbb{R}$ such that~$\bar{\Sigma}^a$ is ISS w.r.t.~$K\times\Lambda$.
\end{enumerate}
\end{prop}
As a consequence of \Cref{prop:01,prop:02,prop:03}, we obtain the following result.
\begin{thm}\label{thm:01}
Suppose that \Cref{ass:01,ass:04,ass:05} are satisfied. Fix~$a\in\mathbb{R}_+$ and let~$K\subset\text{\sffamily{\upshape{X}}}$ be nonempty and compact. Then, the following implications hold:
\begin{enumerate}[label=(\alph*)]
	\item If~$\bar{\Sigma}^a_{\text{\sffamily{\upshape{X}}}}$ is 0-GAS w.r.t.~$K$, then there exists some nonempty, compact~$\Lambda\subset\mathbb{R}$ such that~$(\tilde{\Sigma}^{a,\omega}_{\text{ES}})_\omega$ is small-disturbance SGPUAS w.r.t~$K\times\Lambda$.
	\item If~$\bar{\Sigma}^a_{\text{\sffamily{\upshape{X}}}}$ is ISS w.r.t.~$K$, then there exists some nonempty, compact~$\Lambda\subset\mathbb{R}$ such that~$(\tilde{\Sigma}^{a,\omega}_{\text{ES}})_\omega$ is large-disturbance SGPUAS w.r.t.~$K\times\Lambda$.
\end{enumerate}
\end{thm}
Note that \Cref{thm:01} provides robust stability properties of the pull-back system~\cref{eq:21}. The change of coordinates in~\cref{eq:23} then allows us to draw conclusions about the solutions of closed-loop system~\cref{eq:20}.


\section{Examples}\label{sec:06}
To illustrate the approach in the previous section, we consider two particular examples of the control system~\cref{eq:14}: a single integrator (\Cref{sec:0601}) and a kinematic unicycle (\Cref{sec:0602}). Because of \Cref{thm:01,rem:07}, we may predict the behavior of the closed-loop system~\cref{eq:20} if we can prove that the averaged system~\cref{eq:31} is 0-GAS or ISS.

A proof of stability for the averaged system~\cref{eq:31} naturally requires suitable assumptions on the objective function~$\psi$. It turns out that we can use the same assumptions for both the single integrator and the kinematic unicycle. In each case, the smooth real-valued objective function~$\psi$ is defined on a Euclidean space~$\mathbb{R}^N$ of appropriate dimension~$N\in\mathbb{N}$. Before we state the assumptions on~$\psi$, we introduce the following notation. For every~$y\in\mathbb{R}$, let~$\psi^{-1}(\geq{y})$ denote the \emph{$y$-superlevel set of~$\psi$}; i.e., the (possibly empty) set of all~$p\in\mathbb{R}^{N}$ with~$\psi(p)\geq{y}$. For every~$p\in\mathbb{R}^N$, let~$\nabla\psi(p)$ denote the \emph{gradient vector} of~$\psi$ at~$p$, let~$\nabla^2\psi(p)$ denote the \emph{Hessian matrix} of~$\psi$ at~$p$, and let~$|\nabla^2\psi(p)|$ denote the \emph{operator norm} of~$\nabla^2\psi(p)$ induced by the Euclidean norm. For given~$p_\ast\in\mathbb{R}^N$, we state the subsequent conditions, which will ensure that~$\bar{\Sigma}^a_{\text{\sffamily{\upshape{X}}}}$ is 0-GAS w.r.t.~$p_\ast$.
\begin{assum}\label{ass:07}~$~$
\begin{enumerate}[label=\arabic*]
	\item\label{cond:1} For every~$p\in\mathbb{R}^{N}$ with~$p\neq{p_\ast}$:~$\psi(p)<y_\ast:=\psi(p_\ast)$.
	\item\label{cond:2} The Hessian matrix~$\nabla^2\psi(p_\ast)$ is negative definite.
	\item\label{cond:3} There exists~$r_1\in\mathbb{R}_+$ such that~$\psi(p_\ast+v)=\psi(p_\ast-v)$ for every~$v\in\mathbb{R}^N$ with~$|v|\leq{r_1}$.
	\item\label{cond:4} There exists~$c_1\in\mathbb{R}_+$ such that~$|\nabla^2\psi(p)|\leq{c_1}$ for every~$p\in\mathbb{R}^N$.
	\item\label{cond:5} For every~$p\in\mathbb{R}^{N}$ with~$p\neq{p_\ast}$:~$\nabla\psi(p)\neq0$.
	\item\label{cond:6} For every~$y\in(\underline{y},y_\ast)$, the set~$\psi^{-1}(\geq{y})$ is compact, where~$\underline{y}:=\inf\{\psi(p) \ | \ p\in\mathbb{R}^N\} \in \mathbb{R}\cup\{-\infty\}$.
	\item\label{cond:7} There exist~$c_2,r_2,r_3\in\mathbb{R}_+$ such that~$|\nabla^2\psi(p+v)|\leq{c_2}|\nabla\psi(p)|$ for every~$p\in\mathbb{R}^N$ with~$|p-p_\ast|\geq{r_2}$ and every~$v\in\mathbb{R}^N$ with~$|v|\leq{r_3}$.
\end{enumerate}
\end{assum}
\begin{figure*}%
\centering$\begin{matrix}\begin{matrix}\hypertarget{singleIntegrator2resonate}{\text{\footnotesize(a)}}\\\vphantom{M}\\\vphantom{M}\\\vphantom{M}\\\vphantom{M}\\\vphantom{M}\\\vphantom{M}\end{matrix}\quad\begin{matrix}\includegraphics{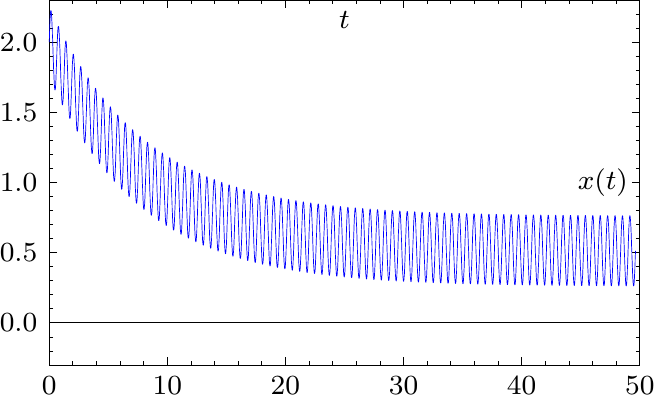}\end{matrix}\qquad&\qquad\begin{matrix}\hypertarget{singleIntegrator2Nonresonate}{\text{\footnotesize(b)}}\\\vphantom{M}\\\vphantom{M}\\\vphantom{M}\\\vphantom{M}\\\vphantom{M}\\\vphantom{M}\end{matrix}\quad\begin{matrix}\includegraphics{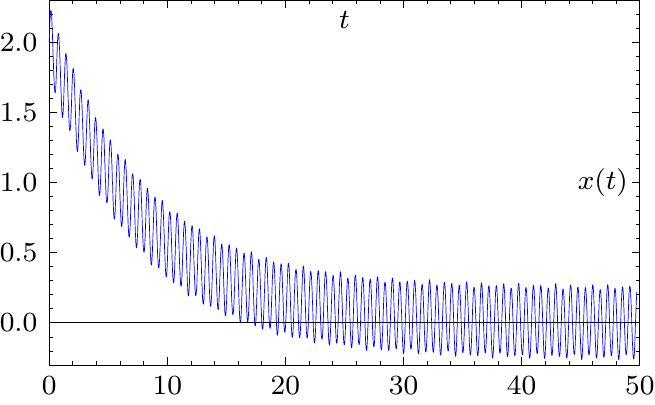}\end{matrix}\end{matrix}$%
\caption{Simulation results for the single integrator~\cref{eq:32} with objective function~$\psi$ as in~\cref{eq:36} and disturbance~$d_{\mathrm{y}}$ as in~\cref{eq:37} for~$a=1$, $\omega=10$, $\varepsilon=0.1$, and different values of~$\omega_{\text{d}}\in\mathbb{R}_+$. \protect\hyperlink{singleIntegrator2resonate}{(a)}:~$\omega_{\text{d}}=\omega$ (resonance); \protect\hyperlink{singleIntegrator2Nonresonate}{(b)}:~$\omega_{\text{d}}=\sqrt{2}\omega$ (no resonance).}%
\label{fig:singleIntegrator}%
\end{figure*}
\begin{rem}
We give some comments on the conditions in \Cref{ass:07} and their intention. Condition~\ref{cond:1} states that~$p_\ast$ is the unique global minimizer of~$\psi$. Condition~\ref{cond:2} ensures that the magnitude of the gradient of~$\psi$ increases sufficiently fast around~$p_\ast$. Condition~\ref{cond:3} states that~$\psi$ is locally even symmetric around~$p_\ast$. In the examples below, this ensures that the remainder vector field in~\cref{eq:26} vanishes at~$p_\ast$. Condition~\ref{cond:4} is used in the examples below to derive a global bound for the remainder vector field in~\cref{eq:26}. Condition~\ref{cond:5} ensures that the gradient vector provides an ascent direction of~$\psi$ at any point~$\neq{p_\ast}$. Condition~\ref{cond:6} is needed to apply a standard Lyapunov argument in the stability analysis for the averaged system. Condition~\ref{cond:7} is trivially satisfied if there exist~$\rho,c\in\mathbb{R}_+$ such that~$|\nabla(p)|\geq{c}$ for every~$p\in\mathbb{R}^N$ with~$|p-p_\ast|\geq\rho$ (because of condition~\ref{cond:4}). If~$\nabla\psi(p)$ vanishes as~$|p|\to\infty$, then condition~\ref{cond:7} provides an additional bound on the remainder vector field in~\cref{eq:26}.
\end{rem}
For given~$p_\ast\in\mathbb{R}^N$, we also state conditions, which will ensure that~$\bar{\Sigma}^a_{\text{\sffamily{\upshape{X}}}}$ is ISS w.r.t.~$p_\ast$.
\begin{assum}\label{ass:08}
Conditions~\ref{cond:1}-\ref{cond:4} in \Cref{ass:07} are satisfied and there exist~$\chi_1,\chi_2,\chi_3\in\mathcal{K}_{\infty}$ such that
\begin{align*}
-\chi_1(|p-p_\ast|) & \ \geq \ \psi(p)-y_\ast \ \geq \ -\chi_2(|p-p_\ast|), \\
|\nabla\psi(p)| & \ \geq \ \chi_3(|p-p_\ast|)
\end{align*}
for every~$p\in\mathbb{R}^{N}$.
\end{assum}
One can easily check that, if \Cref{ass:08} is satisfied, then \Cref{ass:07} is satisfied.


\subsection{Single integrator}\label{sec:0601}
One of the simplest (but nevertheless important) control systems is a single integrator
\begin{equation}\label{eq:32}
\dot{x}(t) \ = \ \mathrm{u}
\end{equation}
on~$\text{\sffamily{\upshape{X}}}:=\mathbb{R}^m$ together with an objective function~$\psi\in\mathfrak{F}(\text{\sffamily{\upshape{X}}})$. Clearly,~\cref{eq:32} is a particular case of~\cref{eq:14} if we set~$F_0$ identically equal to zero and if, for every~$i\in\{1,\ldots,m\}$, the vector field~$F_i$ is identically equal to the~$i$th unit vector of~$\mathbb{R}^m$. Therefore, we can apply the extremum seeking control law~\cref{eq:15} to~\cref{eq:32}, where the oscillatory signals~$u,v\in{L_{\infty}^m}$ are chosen as in \Cref{exm:01}. Then, \Cref{ass:01,ass:04,ass:05} are satisfied, and therefore all definitions and statements in \Cref{sec:05} apply to the single integrator system. In this simple case, the change of coordinates~\cref{eq:23} between solutions~$\xi=[x,\eta]$ of the closed-loop system~\cref{eq:20} and solutions~$\tilde{\xi}=[\tilde{x},\tilde{\eta}]$ of the pull-back system~\cref{eq:21} reduces to
\begin{equation}\label{eq:33}
x(t) \ = \ \tilde{x}(t) + a\,U(\omega{t}), \qquad \eta(t) \ = \ \tilde{\eta}(t),
\end{equation}
where~$U^i(\omega{t})=\alpha^i\sin(\varpi^i\omega{t})$ for every~$i\in\{1,\ldots,m\}$. Moreover, it is easy to check that the right-hand side~\cref{eq:29} of the averaged system~\cref{eq:31} is given by
\begin{equation}\label{eq:34}
\bar{\Sigma}^a_{\text{\sffamily{\upshape{X}}}}(\bar{x},\bar{w}) \ = \ \overline{Uv}\cdot\nabla\psi(\bar{x}) + \delta\bar{G}^{a}(\bar{x}) + \bar{w},
\end{equation}
where~$\overline{Uv}$ is the positive definite~$(m\times{m})$-diagonal matrix with entries~$\overline{Uv}^{i,j}$ given by~\cref{eq:10}, and the remainder vector field~$\delta\bar{G}^{a}$ is given by~\cref{eq:26}. A direct computation, using~\cref{eq:33}, shows that
\begin{subequations}\label{eq:35}
\begin{align}
& \delta\bar{G}^{a}(\bar{x}) \ = \ \frac{a}{T}\int_0^T\int_0^1(1-s) \\
& \quad\times\langle{\nabla^2\psi(\bar{x}+asU(\tau))U(\tau),U(\tau)}\rangle\,v(\tau)\,\mathrm{d}s\,\mathrm{d}\tau
\end{align}
\end{subequations}
for every~$a\in\mathbb{R}_+$ and every~$\bar{x}\in\text{\sffamily{\upshape{X}}}$.
\begin{rem}\label{rem:10}
Note that the sinusoids~$u,v$ in \Cref{exm:01} can always be chosen in such a way that the matrix~$\overline{Uv}$ in~\cref{eq:34} is the identity. Moreover, the remainder term~\cref{eq:35} vanishes as~$a\to0$. Consequently, if~$a$ is sufficiently small, then the averaged system is more or less the same as
\[
\dot{\bar{x}}(t) \ = \ \nabla\psi(\bar{x}(t)) + \bar{w}(t).
\]
This indicates that the closed-loop system approximates the behavior of the gradient system of~$\psi$ with a disturbance $\bar{w}(t)=[w_{a,\omega{t},d(t)}^1,\ldots,w_{a,\omega{t},d(t)}^m]^{\top}$ given by~\cref{eq:19A}.
\end{rem}
We can prove the following results for the averaged system~\cref{eq:31} with its right-hand side~$\bar{\Sigma}^a_{\text{\sffamily{\upshape{X}}}}$ given by~\cref{eq:34}.
\begin{thm}\label{thm:02}
Let~$x_\ast\in\mathbb{R}^m$.
\begin{enumerate}[label=(\alph*)]
	\item Suppose that \Cref{ass:07} is satisfied for~$N=m$ and~$p_\ast=x_\ast$. Then, there exists~$a_0\in\mathbb{R}_+$ such that~$\bar{\Sigma}^a_{\text{\sffamily{\upshape{X}}}}$ is 0-GAS w.r.t.~$x_\ast$ for every~$a\in(0,a_0)$.
	\item Suppose that \Cref{ass:08} is satisfied for~$N=m$ and~$p_\ast=x_\ast$. Then, there exists~$a_0\in\mathbb{R}_+$ such that~$\bar{\Sigma}^a_{\text{\sffamily{\upshape{X}}}}$ is ISS w.r.t.~$x_\ast$ for every~$a\in(0,a_0)$.
\end{enumerate}
\end{thm}
The idea of the proof can be found in \Cref{sec:appendix}.

To generate numerical data, we consider a single integrator in dimension~$m:=1$ with a quadratic objective function~$\psi$ given by
\begin{equation}\label{eq:36}
\psi(x) \ := \ -x^2/2.
\end{equation}
It is clear that \Cref{ass:08} is satisfied for~$N=1$ and~$p_\ast=0$. Therefore, we may conclude from \Cref{thm:01,thm:02} that, if~$a\in\mathbb{R}_+$ is sufficiently small, then $(\tilde{\Sigma}_{\text{ES}}^{a,\omega})_{\omega}$ is large-disturbance SGPUAS w.r.t. the origin in the sense of \Cref{def:06}. We choose the constants~$\alpha:=1/4$ and~$c:=\varpi:=1$ for the sinusoids~$u,v$ in \Cref{exm:01} and~$h:=1$ for the gain in~\cref{eq:16}. As an example of a disturbance~$d=[d_{\mathrm{u}}^{\top},d_{\mathrm{y}}]^{\top}$, we consider~$d_{\mathrm{u}}(t):=0$,
\begin{equation}\label{eq:37}
d_{\mathrm{y}}(t) \ := \ \varepsilon\,\operatorname{sgn}(\sin(\omega_{\text{d}}\,t)),
\end{equation}
where~$\varepsilon,\omega_{\text{d}}\in\mathbb{R}_+$ and $\operatorname{sgn}\colon\mathbb{R}\to\{-1,+1\}$ is the sign function. One can see in \Cref{fig:singleIntegrator} that a disturbance with $\omega_{\text{d}}=\omega$ has a particularly negative influence on the closed-loop system because it generates a non-vanishing drift through resonances with the dither signal~$t\mapsto{v}(\omega{t})$. An uncorrelated disturbance has almost no effect on the closed-loop system.


\begin{figure*}%
\centering$\begin{matrix}\begin{matrix}\hypertarget{unicycleTL}{\text{\footnotesize(a)}}\\\vphantom{M}\\\vphantom{M}\\\vphantom{M}\\\vphantom{M}\\\vphantom{M}\\\vphantom{M}\end{matrix}\quad\begin{matrix}\includegraphics{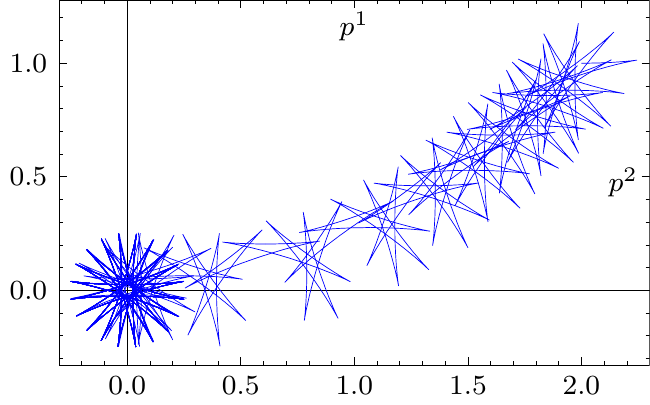}\end{matrix}\qquad&\qquad\begin{matrix}\hypertarget{unicycleTR}{\text{\footnotesize(b)}}\\\vphantom{M}\\\vphantom{M}\\\vphantom{M}\\\vphantom{M}\\\vphantom{M}\\\vphantom{M}\end{matrix}\quad\begin{matrix}\includegraphics{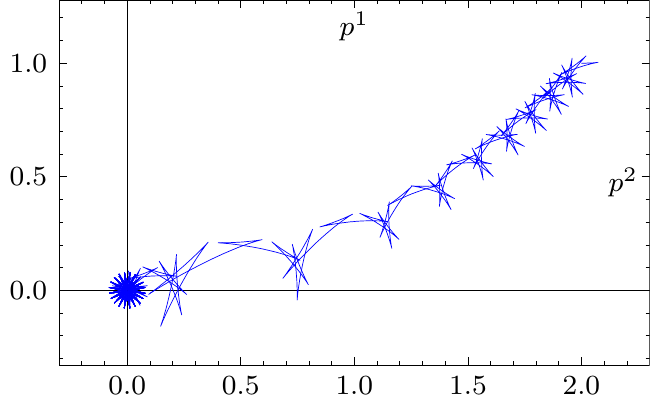}\end{matrix}\\\begin{matrix}\hypertarget{unicycleBL}{\text{\footnotesize(c)}}\\\vphantom{M}\\\vphantom{M}\\\vphantom{M}\\\vphantom{M}\\\vphantom{M}\\\vphantom{M}\end{matrix}\quad\begin{matrix}\includegraphics{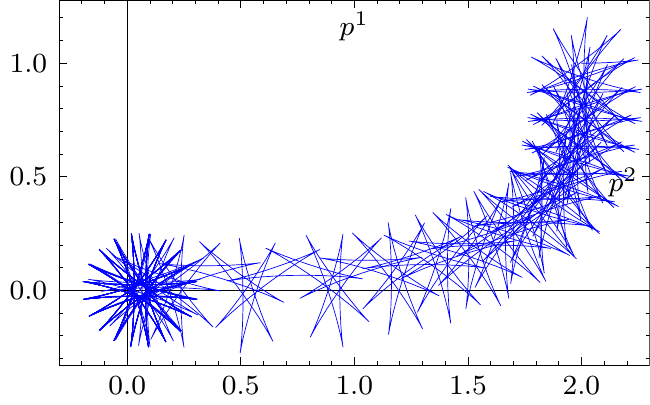}\end{matrix}\qquad&\qquad\begin{matrix}\hypertarget{unicycleBR}{\text{\footnotesize(d)}}\\\vphantom{M}\\\vphantom{M}\\\vphantom{M}\\\vphantom{M}\\\vphantom{M}\\\vphantom{M}\end{matrix}\quad\begin{matrix}\includegraphics{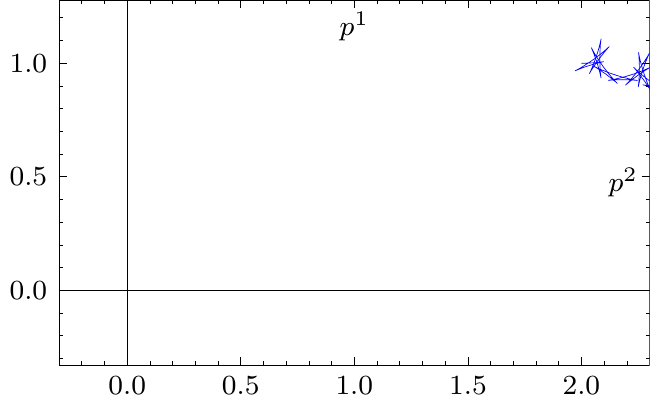}\end{matrix}\end{matrix}$%
\caption{Simulation results for the source seeking kinematic unicycle~\cref{eq:40} with signal function~$\psi$ as in~\cref{eq:44} and disturbance~$d_{\mathrm{y}}$ as in~\cref{eq:45} for~$\omega=10$ and different values of~$a,\varepsilon\in\bar{\mathbb{R}}_+$. \protect\hyperlink{unicycleTL}{(a)}:~$a=1$, $\varepsilon=0$; \protect\hyperlink{unicycleTR}{(b)}:~$a=1/\sqrt{\omega}$, $\varepsilon=0$; \protect\hyperlink{unicycleBL}{(c)}:~$a=1$, $\varepsilon=0.1$; \protect\hyperlink{unicycleBR}{(d)}:~$a=1/\sqrt{\omega}$, $\varepsilon=0.1$. In plot~\protect\hyperlink{unicycleBR}{(d)}, the disturbance leads to an escape of the trajectory with increasing time}%
\label{fig:unicycle}%
\end{figure*}
\subsection{Unicycle}\label{sec:0602}
As a second example, we consider the intensively-studied problem of source seeking with a kinematic unicycle. At any time~$t\in\mathbb{R}$, the state
\begin{equation*}
x(t) \ = \ [ p(t), o(t) ] \ \in \ \mathbb{R}^2\times\mathbb{S}^1 \ =: \ \text{\sffamily{\upshape{X}}}
\end{equation*}
of the unicycle is determined by its current position~$p(t)\in\mathbb{R}^2$ and its current orientation~$o(t)\in\mathbb{S}^1$, where~$\mathbb{S}^1$ denotes the circle in~$\mathbb{R}^2$ of radius~1 centered at the origin. If we choose a standard chart
\begin{equation}\label{eq:38}
[ p^1,p^2,\theta]\colon\mathcal{U}\subset\text{\sffamily{\upshape{X}}}\to\mathbb{R}^3
\end{equation}
for~$\text{\sffamily{\upshape{X}}}$, then the kinematic unicycle model reads
\begin{equation*}
\dot{p}^1(t) \, = \, \mathrm{u}^1\cos\theta(t), \quad \dot{p}^2(t) \, = \, \mathrm{u}^1\sin\theta(t), \quad \dot{\theta}(t) \, = \, \mathrm{u}^0,
\end{equation*}
where~$\mathrm{u}^1$ is the input channel for the forward velocity and~$\mathrm{u}^0$ is the input channel for the angular velocity. We assume that the source signal is given by a smooth real-valued (purely position-dependent) function~$\psi$ on~$\mathbb{R}^2$. We assume that, at any time~$t\in\mathbb{R}$, the unicycle can measure the value
\begin{equation*}
\hat{y}(t) \ = \ \psi(p(t)) + d_{\mathrm{y}}(t)
\end{equation*}
of~$\psi$ at~$p(t)$ up to some measurement noise~$d_{\mathrm{y}}\in{L_{\infty}}$. The goal is to find a position where the signal function attains a maximum value.

To give the unicycle access to any direction of the plane, we apply a constant angular speed
\begin{equation}\label{eq:39}
\mathrm{u}^0 \ = \ \Omega \ \in \ \mathbb{R}_+.
\end{equation}
Next, to establish a situation as in \Cref{sec:05}, we set~$m:=1$ and define two vector fields
\begin{equation*}
F_0 \ := \ \Omega\,\frac{\partial}{\partial\theta} \quad \text{and} \quad F_1 \ := \ \cos\theta\,\frac{\partial}{\partial{p^1}} + \sin\theta\,\frac{\partial}{\partial{p^2}}
\end{equation*}
on~$\text{\sffamily{\upshape{X}}}$ in the coordinates of~\cref{eq:38}. Then, the unicycle control system can be written as the single-input system
\begin{equation}\label{eq:40}
\dot{x}(t) \ = \ F_0(x(t)) + \mathrm{u}^1\,F_1(x(t))
\end{equation}
on~$\text{\sffamily{\upshape{X}}}$, which is a particular case of~\cref{eq:14}. By a slight abuse of notation, we define~$\psi\in\mathfrak{F}(\text{\sffamily{\upshape{X}}})$ by~$\psi(x):=\psi(p)$ for every~$x=[p,o]\in\text{\sffamily{\upshape{X}}}$. Now we are in the situation of \Cref{sec:05} and we can apply the extremum seeking control law~\cref{eq:15} to~\cref{eq:40}, where the oscillatory signals~$u,v\colon\mathbb{R}\to\mathbb{R}$ are chosen as in \Cref{exm:01}. Then, \Cref{ass:01,ass:04,ass:05} are satisfied, and therefore all definitions and statements in \Cref{sec:05} apply to the unicycle system. In this simple case, the change of coordinates~\cref{eq:23} between solutions~$\xi=[p,o,\eta]$ of the closed-loop system~\cref{eq:20} and solutions~$\tilde{\xi}=[\tilde{p},\tilde{o},\tilde{\eta}]$ of the pull-back system~\cref{eq:21} reduces to
\begin{subequations}\label{eq:41}
\begin{align}
p(t) & \ = \ \tilde{p}(t) + a\,U(\omega{t})\,\tilde{o}(t), \\
o(t) & \ = \ \tilde{o}(t), \qquad \eta(t) \ = \ \tilde{\eta}(t),
\end{align}
\end{subequations}
where~$U(\omega{t})=\alpha\sin(\varpi\omega{t})$. Moreover, a direct computation reveals that the right-hand side~\cref{eq:29} of the averaged system~\cref{eq:31} is given by
\begin{subequations}\label{eq:42}
\begin{align}
\bar{\Sigma}^a_{\text{\sffamily{\upshape{X}}}}(\bar{x},\bar{w}) & \ = \ F_0(\bar{x}) + \overline{Uv}\,\langle\nabla\psi(\bar{p}),\bar{o}\rangle\,F_1(\bar{x}) \\
& \qquad + \delta\bar{G}^{a}(\bar{x}) + \bar{w}\,F_1(\bar{x}),
\end{align}
\end{subequations}
where~$\overline{Uv}=\alpha{c}/2$ as in~\cref{eq:10}, and the remainder vector field~$\delta\bar{G}^{a}$ is given by~\cref{eq:26}. A direct computation, using~\cref{eq:41}, shows that
\begin{subequations}\label{eq:43}
\begin{align}
& \delta\bar{G}^{a}(\bar{x}) \ = \ \frac{a}{T}\int_0^T\int_0^1(1-s)\,U(\tau)\,U(\tau)\,v(\tau) \\
& \qquad\times\langle{\nabla^2\psi(\bar{p}+asU(\tau)\bar{o})\bar{o},\bar{o}}\rangle\,\mathrm{d}s\,\mathrm{d}\tau\,F_1(\bar{x})
\end{align}
\end{subequations}
for every~$a\in\mathbb{R}_+$ and every~$\bar{x}=[\bar{p},\bar{o}]\in\text{\sffamily{\upshape{X}}}$. We can prove the following results for the averaged system~\cref{eq:31} with its right-hand side~$\bar{\Sigma}^a_{\text{\sffamily{\upshape{X}}}}$ given by~\cref{eq:42}.
\begin{thm}\label{thm:03}
Let~$p_\ast\in\mathbb{R}^2$.
\begin{enumerate}[label=(\alph*)]
	\item\label{thm:03:a} Suppose that \Cref{ass:07} is satisfied for~$N=2$. Then, there exists~$a_0\in\mathbb{R}_+$ such that~$\bar{\Sigma}^a_{\text{\sffamily{\upshape{X}}}}$ is 0-GAS w.r.t.~$\{p_\ast\}\times\mathbb{S}^1$ for every~$a\in(0,a_0)$.
	\item\label{thm:03:b} Suppose that \Cref{ass:08} is satisfied for~$N=2$. Then, there exists~$a_0\in\mathbb{R}_+$ such that~$\bar{\Sigma}^a_{\text{\sffamily{\upshape{X}}}}$ is ISS w.r.t.~$\{p_\ast\}\times\mathbb{S}^1$ for every~$a\in(0,a_0)$.
\end{enumerate}
\end{thm}
The proof can be found in \Cref{sec:appendix}.

To generate numerical results, we suppose that the source signal function~$\psi$ is given by
\begin{equation}\label{eq:44}
\psi(p) \ = \ \frac{4}{1+(p^1)^2+2\,(p^2)^2}
\end{equation}
for every~$p=[p^1,p^2]\in\mathbb{R}^2$. One can check that \Cref{ass:07} is satisfied for~$N=2$. Therefore, we may conclude from \Cref{thm:01,thm:03} that, if~$a\in\mathbb{R}_+$ is sufficiently small, then $(\tilde{\Sigma}_{\text{ES}}^{a,\omega})_{\omega}$ is small-disturbance SGPUAS w.r.t.~$\{0\}\times\mathbb{S}^1$ in the sense of \Cref{def:05}. We choose the constants~$\alpha:=1/4$ and~$c:=\varpi:=1$ for the sinusoids~$u,v$ in \Cref{exm:01}, $h:=1$ for the gain in~\cref{eq:16}, and~$\Omega:=1$ for the angular speed in~\cref{eq:39}. As an example of a disturbance~$d=[d_{\mathrm{u}}^{\top},d_{\mathrm{y}}]^{\top}$, we consider~$d_{\mathrm{u}}(t):=0$,
\begin{equation}\label{eq:45}
d_{\mathrm{y}}(t) \ := \ \varepsilon\,\sin(\omega{t})\,\cos(\Omega{t}),
\end{equation}
where~$\varepsilon\in\bar{\mathbb{R}}_+$ is an amplitude and~$\omega\in\mathbb{R}_+$ is the frequency of the oscillatory signals. This type of disturbance has a particularly negative influence on the closed-loop system because it generates a non-vanishing drift through resonances with the dither signal~$t\mapsto{v}(\omega{t})$. \Cref{fig:unicycle} contains a selection of trajectories of the closed-loop system~\cref{eq:20}. One can see that an increase of the parameter~$a\in\mathbb{R}_+$ leads to stronger robustness against disturbances but also to less accuracy if no disturbances are present.

\bibliographystyle{plain}
\bibliography{bibFile}

\begin{thebibliography}{10}

\bibitem{AriyurBook}
Kartik~B. Ariyur and Miroslav Krsti{\'c}.
\newblock {\em {Real-Time Optimization by Extremum-Seeking Control}}.
\newblock John Wiley \& Sons, Inc., Hoboken, NJ, 2003.

\bibitem{BulloBook}
Francesco Bullo and Andrew~D. Lewis.
\newblock {\em {Geometric Control of Mechanical Systems}}, volume~49 of {\em
  Texts in Applied Mathematics}.
\newblock Springer, New York, 2005.

\bibitem{Drakunov1995}
Sergey Drakunov, \"Umit \"Ozg\"uner, Peter Dix, and Behrouz Ashrafi.
\newblock {ABS control using optimum search via sliding modes}.
\newblock {\em IEEE Transactions on Control Systems Technology}, 3(1):79--85,
  1995.

\bibitem{Duerr2013}
Hans-Bernd D\"urr, Milo{\u{s}}~S. Stankovi{\'c}, Christian Ebenbauer, and
  Karl~H. Johansson.
\newblock {Lie Bracket Approximation of Extremum Seeking Systems}.
\newblock {\em Automatica}, 49(6):1538--1552, 2013.

\bibitem{Fu2011}
Lina Fu and Ümit Özgüner.
\newblock {Extremum seeking with sliding mode gradient estimation and
  asymptotic regulation for a class of nonlinear systems}.
\newblock {\em Automatica}, 47(12):2595--2603, 2011.

\bibitem{Guay2003}
Martin Guay and Tao Zhang.
\newblock {Adaptive extremum seeking control of nonlinear dynamic systems with
  parametric uncertainties}.
\newblock {\em Automatica}, 39(7):1283--1293, 2003.

\bibitem{Jiang1994}
Zhong-Ping Jiang, Andrew~R. Teel, and Laurent Praly.
\newblock {Small-gain theorem for ISS systems and applications}.
\newblock {\em Mathematics of Control, Signals and Systems}, 7(2):95--120,
  1994.

\bibitem{Khong20132}
Sei~Z. Khong, Dragan Ne{\v{s}}i{\'{c}}, Chris Manzie, and Ying Tan.
\newblock {Multidimensional global extremum seeking via the DIRECT optimisation
  algorithm}.
\newblock {\em Automatica}, 49(7):1970--1978, 2013.

\bibitem{Krstic2000}
Miroslav Krsti{\'c} and Hsin-Hsiung Wang.
\newblock {Stability of extremum seeking feedback for general nonlinear dynamic
  systems}.
\newblock {\em Automatica}, 36(4):595--601, 2000.

\bibitem{Leyva2006}
Ramon Leyva, Corinne Alonso, Isabelle Queinnec, Angel Cid-Pastor, Denis
  Lagrange, and Luis Martinez-Salamero.
\newblock {MPPT of photovoltaic systems using extremum-seeking control}.
\newblock {\em IEEE Transactions on Aerospace and Electronic Systems},
  42(1):249--258, 2006.

\bibitem{LiuBook}
Shu-Jun Liu and Miroslav Krsti\'c.
\newblock {\em {Stochastic Averaging and Stochastic Extremum Seeking}}.
\newblock Communications and Control Engineering. Springer, London, 2012.

\bibitem{Moreau2000}
Luc Moreau and Dirk Aeyels.
\newblock {Practical stability and stabilization}.
\newblock {\em IEEE Transactions on Automatic Control}, 45(8):1554--1558, 2000.

\bibitem{Pan2003}
Yaodong Pan, \"Umit \"Ozg\"uner, and Tankut Acarman.
\newblock {Stability and performance improvement of extremum seeking control
  with sliding mode}.
\newblock {\em International Journal of Control}, 76(9-10):968--985, 2003.

\bibitem{Popovic2006}
Dobrivoje Popovic, Mrdjan Jankovic, Steve Magner, and Andrew~R. Teel.
\newblock {Extremum seeking methods for optimization of variable cam timing
  engine operation}.
\newblock {\em IEEE Transactions on Control Systems Technology},
  14(3):398--407, 2006.

\bibitem{Scheinker2013}
Alexander Scheinker and Miroslav Krsti{\'c}.
\newblock {Minimum-Seeking for CLFs: Universal Semiglobally Stabilizing
  Feedback Under Unknown Control Directions}.
\newblock {\em IEEE Transactions on Automatic Control}, 58(5):1107--1122, 2013.

\bibitem{ScheinkerBook}
Alexander Scheinker and Miroslav Krsti{\'c}.
\newblock {\em {Model-Free Stabilization by Extremum Seeking}}.
\newblock Springer Briefs in Control, Automation and Robotics. Springer, Cham,
  2017.

\bibitem{Scheinker2016}
Alexander Scheinker and David Scheinker.
\newblock {Bounded extremum seeking with discontinuous dithers}.
\newblock {\em Automatica}, 69:250--257, 2016.

\bibitem{Sontag1995}
Eduardo~D. Sontag and Yuan Wang.
\newblock {On Characterizations of Input-to-State Stability with Respect to
  Compact Sets}.
\newblock In {\em Proceedings of the 3rd IFAC Symposium on Nonlinear Control
  Systems Design}, pages 203--208, 1995.

\bibitem{Stankovic2010}
Milo{\u{s}}~S. Stankovi{\'c} and Du{\u{s}}an~M. Stipanovi{\'c}.
\newblock {Extremum seeking under stochastic noise and applications to mobile
  sensors}.
\newblock {\em Automatica}, 46(8):1243--1251, 2010.

\bibitem{Tan2010}
Ying Tan, Will~H. Moase, Chris Manzie, Dragan Ne{\v{s}}i{\'{c}}, and Iven M.~Y.
  Mareels.
\newblock {Extremum seeking from 1922 to 2010}.
\newblock In {\em {Proceedings of the 29th Chinese Control Conference}}, pages
  14--26, 2010.

\bibitem{Tan2006}
Ying Tan, Dragan Ne{\v{s}}i{\'{c}}, and Iven Mareels.
\newblock {On non-local stability properties of extremum seeking control}.
\newblock {\em Automatica}, 42(6):889--903, 2006.

\bibitem{Wang1999}
Hsin-Hsiung Wang, Miroslav Krsti{\'c}, and Georges Bastin.
\newblock {Optimizing bioreactors by extremum seeking}.
\newblock {\em International Journal of Adaptive Control and Signal
  Processing}, 13(8):651--669, 1999.

\bibitem{Ye2016}
Maojiao Ye and Guoqiang Hu.
\newblock {A robust extremum seeking scheme for dynamic systems with
  uncertainties and disturbances}.
\newblock {\em Automatica}, 66:172--178, 2016.

\bibitem{Zhang20071}
Chunlei Zhang, Daniel Arnold, Nima Ghods, Antranik Siranosian, and Miroslav
  Krsti{\'c}.
\newblock {Source seeking with non-holonomic unicycle without position
  measurement and with tuning of forward velocity}.
\newblock {\em Systems \& Control Letters}, 56(3):245--252, 2007.

\bibitem{Zhang2009}
Chunlei Zhang and Ra{\'{u}}l Ord{\'{o}}{\~{n}}ez.
\newblock Robust and adaptive design of numerical optimization-based extremum
  seeking control.
\newblock {\em Automatica}, 45(3):634--646, 2009.

\bibitem{ZhangBook}
Chunlei Zhang and Ra{\'{u}}l Ord{\'{o}}{\~{n}}ez.
\newblock {\em {Extremum-Seeking Control and Applications}}.
\newblock Advances in Industrial Control. Springer, London, 2012.

\bibitem{Zhang20072}
Chunlei Zhang, Antranik Siranosian, and Miroslav Krsti{\'c}.
\newblock {Extremum seeking for moderately unstable systems and for autonomous
  vehicle target tracking without position measurements}.
\newblock {\em Automatica}, 43(10):1832--1839, 2007.

\end{thebibliography}


\appendix
\section{Appendix}\label{sec:appendix}
We only give a proof of \Cref{thm:03}. The statements of \Cref{thm:02} can be deduced in a similar way. Throughout this section, we use the same notation as in \Cref{sec:0602} and we suppose that \Cref{ass:07} is satisfied.

For every~$o\in\mathbb{S}^1$, let~$o_{\perp}$ denote the unique element of~$\mathbb{S}^1$ such that~$(o,o_{\perp})$ is a positively oriented orthogonal basis of~$\mathbb{R}^2$ with respect to~$\langle\cdot,\cdot\rangle$. Next, we introduce a candidate for a Lyapunov function. For every~$\varepsilon\in\mathbb{R}_+$, define~$V_{\varepsilon}\colon{\text{\sffamily{\upshape{X}}}}\to\mathbb{R}$ by
\[
V_{\varepsilon}(x) \ := \ y_\ast - \psi(p) + \varepsilon\,\langle\nabla\psi(p),o-o_{\perp}\rangle^2
\]
for every~$x=[p,o]\in\text{\sffamily{\upshape{X}}}$. For every~$\varepsilon\in\mathbb{R}_+$ and every~$z\in\mathbb{R}$, let~$V_{\varepsilon}^{-1}(\leq{z})$ denote the \emph{$z$-sublevel set of~$V_{\varepsilon}$}; i.e., the (possibly empty) set of all~$x\in\text{\sffamily{\upshape{X}}}$ with~$V_{\varepsilon}(x)\leq{z}$.
\begin{lem}\label{lem:01}
Define~$z_\ast:=y_\ast-\underline{y}\in\mathbb{R}_+\cup\{+\infty\}$. Then, there exists~$\varepsilon_1\in\mathbb{R}_+$ such that, for every~$\varepsilon\in(0,\varepsilon_1)$,
\begin{itemize}
	\item~$0<V_{\varepsilon}(x)<z_\ast$ for every~$x=[p,o]\in\text{\sffamily{\upshape{X}}}$ with~$p\neq{p_\ast}$,
	\item~$V_{\varepsilon}^{-1}(\leq{z})$ is compact for every~$z\in(0,z_\ast)$.
\end{itemize}
\end{lem}
\begin{proof}
In the proof, we use conditions~\ref{cond:1} and~\ref{cond:4}-\ref{cond:6} in \Cref{ass:07}. A Taylor expansion of~$\psi$ at~$p\pm{v}$ around~$p\in\mathbb{R}^2$, using condition \ref{cond:4} for the remainder, leads to the estimates~$\mp\psi(p\pm{v})\leq\mp\psi(p)-c_1|v|^2/2$, where~$v:=\nabla\psi(p)/c_1$. Note that~$\psi(p+v)\leq{y_\ast}$ and~$\underline{y}\leq\psi(p-v)$ by conditions~\ref{cond:1} and~\ref{cond:6}, respectively. It follows that
\[
|\nabla\psi(p)|^2\leq2{c_1}(\psi(p)-\underline{y}),\quad |\nabla\psi(p)|^2\leq2{c_1}(y_\ast-\psi(p))
\]
for every~$p\in\mathbb{R}^2$. Choose~$\varepsilon_1\in\mathbb{R}_+$ such that
\[
1/2 \ < 1-4c_1\varepsilon_1 \ < \ 1+4c_1\varepsilon_1 \ < \ 2
\]
and fix an arbitrary~$\varepsilon\in(0,\varepsilon_1)$. Using the Cauchy-Schwarz inequality, we conclude that
\begin{align}
(y_\ast-\psi(p))/2 & \ \leq \ V_{\varepsilon}(x) \ \leq \ 2\,(y_\ast-\psi(p)), \label{eq:A1} \\
(\psi(p)-\underline{y})/2 & \ \leq \ z_\ast-V_{\varepsilon}(x) \ \leq \ 2\,(\psi(p)-\underline{y}) \nonumber
\end{align}
for every~$x=[p,o]\in\text{\sffamily{\upshape{X}}}$. Now, using conditions~\ref{cond:1},~\ref{cond:5}, and~\ref{cond:6}, it is easy to check that the assertions are true.
\end{proof}
As an abbreviation, for each~$a\in\mathbb{R}$, define~$\phi^a\colon{\text{\sffamily{\upshape{X}}}}\to\mathbb{R}$ by
\begin{align*}
\phi^a(x) & \ := \ \frac{1}{T}\int_0^T\int_0^1(1-s)\,U(\tau)\,U(\tau)\,v(\tau) \\
& \qquad\qquad\times\langle{\nabla^2\psi(p+asU(\tau)o)o,o}\rangle\,\mathrm{d}s\,\mathrm{d}\tau.
\end{align*}
For every~$w\in\mathbb{R}$, let~$x\mapsto(\bar{\Sigma}^a_{\text{\sffamily{\upshape{X}}}}V_{\varepsilon})(x,w)$ denote the Lie derivative of~$V_{\varepsilon}$ along the vector field~$x\mapsto\bar{\Sigma}^a_{\text{\sffamily{\upshape{X}}}}(x,w)$.
\begin{lem}\label{lem:02}
There exists~$\varepsilon_2\in\mathbb{R}_+$ such that
\[
(\bar{\Sigma}^a_{\text{\sffamily{\upshape{X}}}}V_{\varepsilon})(x,w) \leq (2a|\phi^a(x)|+2|w|-\varepsilon\Omega|\nabla\psi(p)|)|\nabla\psi(p)|
\]
for every~$\varepsilon\in(0,\varepsilon_2)$, every~$a\in\mathbb{R}_+$, every~$w\in\mathbb{R}$, and every~$x=[p,o]\in\text{\sffamily{\upshape{X}}}$ with~$\Omega\in\mathbb{R}_+$ as in~\cref{eq:39}.
\end{lem}
\begin{proof}
Following the idea in the proof of Theorem~6.45 in~\cite{BulloBook}, a lengthy but direct computation reveals that
\begin{align*}
(\bar{\Sigma}^a_{\text{\sffamily{\upshape{X}}}}V_{\varepsilon})(x,w) & \ = \ (a\phi^{a}(x)+w)\,\langle{b_{\varepsilon}(x),\nabla\psi(p)}\rangle \\
& \qquad - \langle{Q_{\varepsilon}(x)\nabla\psi(p),\nabla\psi(p)\rangle}
\end{align*}
for all~$\varepsilon,a\in\mathbb{R}_+$, every~$x=[p,o]\in\text{\sffamily{\upshape{X}}}$, and every~$w\in\mathbb{R}$ with a certain symmetric matrix~$Q_{\varepsilon}(x)\in\mathbb{R}^{2\times2}$ and a certain vector~$b_{\varepsilon}(x)\in\mathbb{R}^2$. Using condition~\ref{cond:4} in \Cref{ass:07}, one can prove that there exists~$\varepsilon_2\in\mathbb{R}_+$ such that, for every~$\varepsilon\in(0,\varepsilon_2)$ and every~$x\in\text{\sffamily{\upshape{X}}}$, the smallest eigenvalue of~$Q_{\varepsilon}(x)$ is~$\geq\varepsilon\Omega$ and the norm of~$b_{\varepsilon}(x)$ is~$\leq{2}$.
\end{proof}
\begin{lem}\label{lem:03}
There exist~$a_0',\kappa\in\mathbb{R}_+$ such that
\[
|\phi^a(x)| \ \leq \ \kappa\,|\nabla\psi(p)|
\]
for every~$a\in[0,a_0']$ and every~$x=[p,o]\in\text{\sffamily{\upshape{X}}}$.
\end{lem}
\begin{proof}
In the proof, we use conditions~\ref{cond:1}-\ref{cond:3}, \ref{cond:5}, and~\ref{cond:7} in \Cref{ass:07}. Conditions~\ref{cond:1} and~\ref{cond:2} ensure the existence of~$\kappa_1',\rho_1\in\mathbb{R}_+$ such that~$|p-p_\ast|\leq\kappa_1'|\nabla\psi(p)|$ for every~$p\in\mathbb{R}^2$ with~$|p-p_\ast|\leq\rho_1$. It follows from condition~\ref{cond:3} that there exists~$a_1\in\mathbb{R}_+$ such that~$\phi^a(x_\ast)=0$ for every~$a\in[0,a_1]$ and every~$x_\ast\in\{p_\ast\}\times\mathbb{S}^1$. Since~$(a,x)\mapsto\phi^a(x)$ is smooth, we may conclude that there exists~$\kappa_1''\in\mathbb{R}_+$ such that~$|\phi^a(x)|\leq\kappa_1''|p-p_\ast|$ for every~$a\in[0,a_1]$ and every~$x=[p,o]\in\text{\sffamily{\upshape{X}}}$ with~$|p-p_\ast|\leq\rho_1$. Consequently, $|\phi^a(x)|\leq\kappa_1|\nabla\psi(p)|$ for every~$a\in[0,a_1]$ and every~$x=[p,o]\in\text{\sffamily{\upshape{X}}}$ with~$|p-p_\ast|\leq\rho_1$, where~$\kappa_1:=\kappa_1'\kappa_1''\in\mathbb{R}_+$. It follows from condition~\ref{cond:7} and the definition of~$\phi^a$ that there exist~$a_2,\kappa_2,\rho_2\in\mathbb{R}_+$ such that~$|\phi^a(x)|\leq\kappa_2|\nabla\psi(p)|$ for every~$a\in[0,a_2]$ and every~$x=[p,o]\in\text{\sffamily{\upshape{X}}}$ with~$|p-p_\ast|\geq\rho_2$. Let~$a_0':=\min\{a_1,a_2\}\in\mathbb{R}_+$. It follows from condition~\ref{cond:5} that there exists~$\kappa_3\in\mathbb{R}_+$ such that~$|\phi^a(x)|\leq\kappa_3|\nabla\psi(p)|$ for every~$a\in[0,a_0']$ and every~$x=[p,o]\in\text{\sffamily{\upshape{X}}}$ with~$\rho_1\leq|p-p_\ast|\leq\rho_2$. The claim follows if we define~$\kappa:=\max\{\kappa_1,\kappa_2,\kappa_3\}\in\mathbb{R}_+$.
\end{proof}
Let~$\varepsilon_1,\varepsilon_2\in\mathbb{R}_+$ as in \Cref{lem:01,lem:02} and define~$\varepsilon_0:=\min\{\varepsilon_1,\varepsilon_2\}$. For the rest of the appendix, fix an arbitrary~$\varepsilon\in(0,\varepsilon_0)$ and abbreviate~$V_{\varepsilon}$ by~$V$. Let~$a_0',\kappa\in\mathbb{R}_+$ as in \Cref{lem:03} and define~$a_0:=\min\{a_0',\varepsilon\Omega/(2\kappa)\}$. For the rest of the appendix, fix an arbitrary~$a\in(0,a_0)$. Then, we obtain from \Cref{lem:02,lem:03} that
\[
\dot{V}(x,w)\,:=\,(\bar{\Sigma}^a_{\text{\sffamily{\upshape{X}}}}V)(x,w)\,\leq\,(2|w|-\delta|\nabla\psi(p)|)\,|\nabla\psi(p)|
\]
for every~$x\in\text{\sffamily{\upshape{X}}}$ and every~$w\in\mathbb{R}$, where~$\delta:=\varepsilon\Omega/2\in\mathbb{R}_+$.

\subsection{Proof of 0-GAS}
We know that~$V=V_{\varepsilon}$ has the properties in \Cref{lem:01}. Because of condition~\ref{cond:5} in \Cref{ass:07}, $\dot{V}(x,0)=-\delta|\nabla\psi(p)|^2<0$ for every~$x=[p,o]\in\text{\sffamily{\upshape{X}}}$ with~$p\neq{p_\ast}$. Now, a standard Lyapunov argument can be applied to show that~$\dot{x}=\bar{\Sigma}^a_{\text{\sffamily{\upshape{X}}}}(x,0)$ is GAS w.r.t.~$\{p_\ast\}\times\mathbb{S}^1$.

\subsection{Proof of ISS}
Suppose that, additionally, \Cref{ass:08} is satisfied with~$\chi_1,\chi_2,\chi_3\in\mathcal{K}_{\infty}$ as therein. Then, because of~\Cref{eq:A1},
\[
\chi_1(|p-p_\ast|)/2 \ \leq \ V(x) \ \leq \ 2\,\chi_2(|p-p_\ast|)
\]
for every~$x=[p,o]\in\text{\sffamily{\upshape{X}}}$. Moreover, for every~$x=[p,o]\in\text{\sffamily{\upshape{X}}}$ and every~$w\in\mathbb{R}$, the following implication holds:
\[
|p-p_\ast| \ > \ \chi_3^{-1}(2|w|/\delta) \qquad \Longrightarrow \qquad \dot{V}(x,w) \ < \ 0.
\]
It follows that~$V$ is an ISS-Lyapunov function for~$\dot{x}=\bar{\Sigma}^a_{\text{\sffamily{\upshape{X}}}}(x,w)$ with respect to~$\{p_\ast\}\times\mathbb{S}^1$. By Theorem~1 in~\cite{Sontag1995}, we conclude that~$\dot{x}=\bar{\Sigma}^a_{\text{\sffamily{\upshape{X}}}}(x,w)$ is ISS w.r.t.~$\{p_\ast\}\times\mathbb{S}^1$.
\end{document}